\documentclass[article]{amsart}

\usepackage{amssymb}
\usepackage{url}
\usepackage{color}

\newtheorem{theorem}{Theorem}[section]
\newtheorem*{theorem*}{Theorem}

\newtheorem{construction}[theorem]{Construction}

\newtheorem{corollary}[theorem]{Corollary}

\newtheorem{definition}[theorem]{Definition}
\newtheorem{example}[theorem]{Example}

\newtheorem{lemma}[theorem]{Lemma}
\newtheorem{notation}[theorem]{Notation}

\newtheorem{proposition}[theorem]{Proposition}
\newtheorem{remark}[theorem]{Remark}
\newtheorem*{remark*}{Remark}

\DeclareMathOperator{\Cov}{Cov}
\begin{document}

\title{A Formula about W-Operator and Its Application to Hurwitz Number}

\author{Hao Sun}


\maketitle

\begin{abstract}
$W$-operators are differential operators on the polynomial ring. Mironov, Morosov and Natanzon construct the generalized Hurwitz numbers. They use the $W$-operator to prove a formula for the generating function of the generalized Hurwitz numbers. A special example of the $W$-operator is the cut-and-join operator. Goulden and Jackson use the cut-and-join operator to calculate the simple Hurwitz number. In this paper, we study the relation between $W$-operator $W([d])$ and the central elements $K_{1^{n-d}d}$ in $\mathbb{C}S_n$. Based on the relation we find, we give another proof about a differential equation of the generating function of $d$-Hurwitz number.
\end{abstract}

\section{Introduction}
The Hurwitz enumeration problem \cite{Lando} aims at classifying all $n$-fold coverings of $S^2$ (or $\mathbb{C}P^1$) with $k$ branched points $\{z_1,...,z_k\}$. Given such a covering, each branched point $z_i$ corresponds to a unique permutation $\sigma_i$ in $S_n$. Denote $\lambda_i$ the partition corresponding to $\sigma_i$. The number of all connected $n$-coverings with $k$ ordered branched points $z_i, 1 \leq i \leq k$, each of which corresponds to a permutation of type $\lambda_i, 1 \leq i \leq k$, is finite. This number is denoted by $\Cov_n(\lambda_1,...,\lambda_k)$. Alternatively, $\Cov_n(\lambda_1,...,\lambda_k)$ is the number of $k$-tuples $(\sigma_1,...,\sigma_k) \in S^{k}_n$ satisfying the following conditions \cite{Carrel} \cite{Lando},
\begin{enumerate}
	\item[(1)]  $\sigma_i$ is of type $\lambda_i$,
	\item[(2)]  $\sigma_1...\sigma_k=1$,
	\item[(3)]  The subgroup generated by the elements $\{\sigma_1,...,\sigma_k\}$ is transitive.
\end{enumerate}
Given $\alpha$ a partition of $n$, the simple Hurwitz number is
\begin{align*}
	h_k(\alpha)=\Cov_n(1^{n-2}2,...,1^{n-2}2,\alpha).
\end{align*}
It is the number of $(k+1)$-tuples $(\tau_1,...,\tau_k,\sigma^{-1}) \in S^{k+1}_n$ satisfying the following conditions
\begin{enumerate}
	\item[(1)]  $\tau_i$ are transpositions (or of type $1^{n-2}2$), where $1 \leq i \leq k$, and $\sigma^{-1}$ is of type $\alpha$, $\alpha=(\alpha_1,\alpha_2,...)$,
	\item[(2)]  $\tau_1...\tau_k=\sigma$,
	\item[(3)]  the group generated by $\{\tau_1,...,\tau_k\}$ is transitive on the set $\{1,...,n\}$.
\end{enumerate}
The generating function $H$ for simple Hurwitz numbers is 	
\begin{align*}
H(z,p)=H(z,p_1,p_2,...)=\sum_{n \geq 1}\frac{1}{n!}\sum_{k=1}^{\infty}\sum_{\alpha \vdash n}h_k(\alpha)\frac{z^k}{k!}p_{\alpha_1}p_{\alpha_2}... \text{ .}
\end{align*}

The cut-and-join operator $\Delta$ is introduced by Goulden \cite{MR1249468}. The $\Delta$ operator is an infinite sum of differential operators in variables $p_i, i \geq 1$. The formula for $\Delta$ is
\begin{align*}
\Delta= \frac{1}{2}\sum_{i\geq1}\sum_{j\geq1}(ijp_{i+j}\frac{\partial^2}{\partial p_i \partial p_j}+(i+j)p_i p_j \frac{\partial}{\partial p_{i+j}}).
\end{align*}
Goulden proves the following formula,
\begin{equation}\label{eq1}
	\Phi(K_{1^{n-2}2}g)=\Delta \Phi(g),
\end{equation}
where $g$ is any element in the permutation group $S_n$, $K_{1^{n-2}2}$ is the central element of $\mathbb{C}S_n$ corresponding to the partition $(1^{n-2}2)$ and $\Phi$ is a linear map from the group ring $\mathbb{C}S_n$ to the polynomial ring $\mathbb{C}[p_1,p_2,...]$. This formula plays an important role in calculating the simple Hurwitz numbers \cite{MR1396978}. Also, Carrel use this formula to prove the following formula for the generating function $H(z,p)$ of simple Hurwitz numbers \cite{Carrel},
\begin{align*}
\frac{\partial e^{H(z,p)}}{\partial z} = \Delta e^{H(z,p)}.
\end{align*}

Mironov et al .\cite{MR2864467} constructed $W$-operators $W([\lambda])$, where $\lambda$ is a partition of some positive integer. They are differential operators acting on the space $\mathbb{C}[[X_{ij}]]_{i,j \geq 1}$ of formal series in variables $X_{ij}$ $(i,j \geq 0)$, where $X_{ij}$ are coordinate functions on the infinite matrix. A subring of $\mathbb{C}[[X_{ij}]]_{i,j \geq 1}$ is $\mathbb{C}[p_1,p_2,...]$, where $p_k=Tr(X^k)$ and $X=(X_{ij})_{i,j \geq 1}$. A direct calculation shows that $W([2])$ is the cut-and-join operator $\Delta$ on the ring $\mathbb{C}[p_1,p_2,...]$.

In section 2, we review the definition and some properties of $W$-operators.

In section 3, we prove an important property of the $W$-operator.
\begin{theorem*}\textbf{\ref{112}}
	For any $g \in \mathbb{C}S_n$,
	\begin{align*}
		\Phi(K_{1^{n-d}d}g)=W([d]) \Phi(g),
	\end{align*}
where $K_{1^{n-d}d}$ is the central element in $\mathbb{C}S_n$ corresponding to the partition $(1^{n-d}d)$.
\end{theorem*}
This property is very similar to the Equation \eqref{eq1} of the cut-and-join operator.

In section 4, we use permutation groups to give another construction of $W$-operators $W([d])$.

In section 5, we generalize the simple Hurwitz numbers and define a new type of Hurwitz numbers $h^{d}_k(\alpha)=\Cov_d(1^{n-d}d,...,1^{n-d}d,\alpha)$, which is the number of all $n$-coverings with $k+1$ branched points, where $k$ of them correspond to $d$-cycles in $S_n$ and the last one corresponds to a permutation of type $\alpha$. We define the generating functions $H^{[d]}$ for the Hurwitz numbers $h^{[d]}_k$ as
\begin{align*}
		H^{[d]}(z,p)=H^{[d]}(z,p_1,p_2,...)=\sum_{n \geq 1}\frac{1}{n!}\sum_{k=1}^{\infty}\sum_{\alpha \vdash n}h^{[d]}_k(\alpha)\frac{z^k}{k!}\Phi(\alpha) \text{ .}
\end{align*}
Finally, we give another proof of following theorem, which is first proved by Mironov et al. \cite{MR2864467}.
\begin{theorem*}\textbf{\ref{52}}
	\begin{align*}
	\frac{\partial e^{H^{[d]}}}{\partial z} = W([d]) e^{H^{[d]}}.
	\end{align*}
\end{theorem*}
\textbf{Acknowledgements}\\

This work has been written on spring 2015 at the University of Illinois at Urbana Champaign under the supervision of Maarten Bergvelt. It is a pleasure for me to thank my advisor Maarten Bergvelt for having introduced me to the problem and for his constant invaluable help. Also, I want to thank Rinat Kedem for her suggestions.

\section{W-Operator}

\begin{definition}\label{11}
	A variable matrix $X$ is an infinite matrix with variable $X_{ab}$ in the $(a,b)$-entry. Generally, $X:=(X_{ab})_{a,b\geq1}$ and all $X_{ab}$ are assumed to commute with each other.
\end{definition}

\begin{definition}\label{12}
	Define $p_{k}$ to be the trace of $X^{k}$, i.e., $p_{k}=tr(X^{k})$.
\end{definition}

Clearly, $p_{k}$ is a power series in $\mathbb{C}[[X_{ab}]]_{a,b \geq 1}$.

\begin{remark}\label{13}
	If $X$ is a special variable matrix with $X_{ab}=0$, when $a\neq b$, then $p_{k}$ is exactly the power symmetric function $\sum_{i=1}^{\infty} X^{k}_{ii}$.
\end{remark}

\begin{definition}\label{14}
	The operator matrix $D$ is an infinite matrix with $D_{ab}$ in the $(a,b)$-entry, where $D_{ab}=\sum\limits_{c=1}^{\infty} X_{ac} \frac{\partial}{\partial X_{bc}}$.
\end{definition}

\begin{definition}\label{15}
	The normal ordered product of $D_{ab}$ and $D_{cd}$ is
	\begin{align*}
		:D_{ab}D_{cd}:=\sum_{e_1,e_2 \geq 1}X_{ae_{1}}X_{ce_{2}} \frac{\partial}{\partial X_{be_{1}}} \frac{\partial}{\partial X_{de_{2}}}.	
	\end{align*}
	Similarly, the normal ordered product $:\prod\limits_{i=1}^{d}D_{a_i b_i}:$ is
	\begin{align*}
		:\prod\limits_{i=1}^{d}D_{a_i b_i}:	= \sum_{e_1,...,e_d \geq 1} (\prod\limits_{i=1}^{d}X_{a_i e_i}\prod\limits_{i=1}^{d}\frac{\partial}{\partial X_{b_i e_i}}).
	\end{align*}
\end{definition}

\begin{definition}\label{16}
	For any positive integer d, we define the W-operator $W([d])$ as
	\begin{align*}
		W([d]):=\frac{1}{d}:tr(D^{d}):.
	\end{align*}
	For any partition $\lambda=(\lambda_{1},\lambda_{2},...,\lambda_{k})$ of a positive integer d,
	\begin{align*}
		W([\lambda])=:W([\lambda_1])...W([\lambda_m]):.
	\end{align*}
\end{definition}

\begin{definition}\label{17}
	Let $(a_1,...,a_d)$ be an $d$-tuple with integers $a_i \geq 1$. We define the differential operator $D_{(a_1,...,a_d)}$ as
	\begin{align*}
		D_{(a_1,...,a_d)}=:\prod\limits_{i=1}^{d} D_{a_i a_{i+1}}:,
	\end{align*}
	where $a_{d+1}=a_1$. Similarly, we define the monomial $X_{(a_1,...,a_d)}$ as
\begin{align*}
X_{(a_1,...,a_d)}=\prod\limits_{i=1}^{d}X_{a_i a_{i+1}},
\end{align*}
where $a_{d+1}=a_1$.
\end{definition}
With this new notation $D_{(a_1,...,a_d)}$, we can write $W([d])$ in the following form
	\begin{align*}
		W([d])=\frac{1}{d} \sum_{a_1,...,a_d \geq 1} D_{(a_1,...,a_d)}.
	\end{align*}	

\begin{theorem}\label{18}
	$W([d])$ is a well-defined operator on $\mathbb{C}[p_{1},p_{2},...]$. In another words, $:tr(D^{d}):F(p)\in\mathbb{C}[p_{1},p_{2},...]$, for any $F(p) \in \mathbb{C}[p_{1},p_{2},...]$.
\end{theorem}

\begin{proof}
	This theorem is proved in \cite{Sun1}, Theorem 3.15.
\end{proof}

\begin{definition}\label{19}
	Define a map
	\begin{align*}
	\Phi: \mathbb{C}S_n \rightarrow \mathbb{C}[p_1,p_2,...]
	\end{align*}
	such that for each $\sigma \in S_n$, we have
	\begin{align*}
		\Phi(\sigma)=p_{\alpha}=p_{\alpha_1}...p_{\alpha_m},
	\end{align*}
	where $\alpha=(\alpha_1,...,\alpha_m)$ is the partition (or type) corresponding to $\sigma$.
\end{definition}

\begin{definition}\label{110}
	Let $\alpha$ be a partition of a positive integer $n$. We define the element $K_\alpha$ in the group ring $\mathbb{C}S_n$ as
	\begin{align*}
		K_\alpha = \sum_{\sigma \in S_n, \atop \sigma \text{ is of type } \alpha}\sigma.
	\end{align*}
\end{definition}

$K_\alpha$ is in the center of in $\mathbb{C}S_n$. For example, $\Phi(K_\alpha)=|K_{\alpha}|p{_\alpha}$, where $|K_\alpha|$ is the number of all $\sigma \in S_n$ of type $\alpha$.

\begin{notation}\label{111}
	Given a partition $\alpha=(\alpha_1,...,\alpha_m)$ of a positive integer $n$, we can write it as
	\begin{align*}
		\alpha=1^{k_1}2^{k_2}...s^{k_s}
	\end{align*}	
	where $k_i$ is the number of times the integer $i$ appears in the partition $\lambda$. For example, if $\lambda=1^{n-d}d$, then $K_{1^{n-d}d}$ is a central element in $\mathbb{C}S_n$, which is the sum of all $d$-cycles in $S_n$.
\end{notation}

\begin{theorem}\label{112}
	For any $g \in \mathbb{C}S_n$, we have
	\begin{align*}
		\Phi(K_{1^{n-d}d}g)=W([d]) \Phi(g).
	\end{align*}
\end{theorem}

We will prove this theorem in the next section.

\section{Proof of Theorem \ref{112}}
\begin{definition}\label{21}
	A quiver $Q=(V,A,s,t)$ is a quadruple, where $V$ is the set of vertices, $A$ is the set of arrows, $s$ and $t$ are two maps $A \to Q$. If $a \in A$, $s(a)$ is the source of this arrow and $t(a)$ is the target of the arrow. We assume that $V$ and $A$ are finite sets.
	
	If $B$ is a subset of $A$, $V_{B}=\{s(a),t(a),a \in B\}$, then we call $(V_{B},B,s',t')$ subquiver of $Q$, where $s'=s|_B$, $t'=t|_B$.
	
	A quiver $Q=(V,A,s,t)$ is connected if the underlying undirected graph of $Q$ is connected.
	
	A connected quiver $Q=(V,A,s,t)$ is a loop, if for any vertex $v \in V$, there is a unique arrow $a \in A$ such that $s(a)=v$ and a unique arrow $b \in A$ such that $t(b)=v$.
\end{definition}

\begin{definition}\label{22}
	Denote by $\mathbb{FQ}$ the set of all quivers $(V,A,s,t)$ with finite vertex set $\{1,...,n\}$ for some positive integer $n$ and finite arrow set $A$. Denote by $\mathbb{M}$ the set of all monomials with variables $X_{ij}$, $1 \leq i,j < \infty$.
\end{definition}

\begin{remark}\label{23}
	Let $Q=(V,A,s,t) \in \mathbb{FQ}$. We define the map $\beta : \mathbb{FQ} \rightarrow \mathbb{M}$ as $\beta(Q)=M_Q$, where $M_Q= \prod_{a \in A} X_{s(a)t(a)}$. Also, given any monomial
\begin{align*}
M=\prod_{k=1}^{l}=X_{i_k j_k},
\end{align*}
we can define the corresponding quiver $Q_M$ as $Q_M=(V_M,A_M,s,t)$, where $V_M=\{1,...,n\},n=max\{i_k,j_k, 1 \leq k \leq n \}$ and $A_M=\{a_k:i_k \rightarrow j_k, 1 \leq k \leq n\}$.
\end{remark}

\begin{definition}\label{24}
	Let $\Phi_n: S_n \rightarrow \mathbb{FQ}$ be the map such that $\Phi_n(\alpha)=Q_\alpha$, where $Q_\alpha=\{ V_\alpha=\{1,...,n\},A_\alpha=\{i \rightarrow \alpha(i) \},s,t \}$, where $s,t$ are the obvious source and target maps. The image of $\Phi_n$ consist of unions of disjoint loops which represent elements of $S_n$. Denote by $M_\alpha=\beta(Q_\alpha)$ the monomial corresponding to $\alpha$.
\end{definition}

\begin{example}\label{25}
Let $\alpha=(123)(45)(6) \in S_6$, then $Q_\alpha=\Phi_n(\alpha)$ is
\begin{align*}
		& 1 \rightarrow 2 \rightarrow 3 \rightarrow 1\\
		& 4 \rightarrow 5 \rightarrow 4\\
		& 6 \rightarrow 6.
\end{align*}
\end{example}

\begin{construction}\label{26}
	Given $\alpha \in S_n$, we have $Q_\alpha=\Phi_n(\alpha)$ the quiver corresponding to $\alpha$. Given two vertices $a_1,a_2 \in Q_\alpha$, first we pick the unique arrow $a_2 \rightarrow b$ with source $a_2$ in $Q_\alpha$, then we use another arrow $a_1 \rightarrow b$ substituting $a_2 \rightarrow b$. So we get a new quiver denoted by $(\bar{D}_{a_1 a_2})Q_{\alpha}$. More generally, if $a_1,...,a_d$ are distinct vertices (or integers) of $Q_\alpha$, we replace the arrows $a_i \rightarrow b_{i}$ with $a_{i-1} \rightarrow b_i$ simultaneously, $2 \leq i \leq k+1$, $a_{k+1}=a_1$. Denote by $(\prod_{i=1}^{d}\bar{D}_{a_{i}a_{i+1}})Q_\alpha$ the new quiver. We introduce another notation similar to Definition \ref{17},
	\begin{align*}
		\bar{D}_{(a_1,...a_d)}=\prod_{i=1}^{d}\bar{D}_{a_i a_{i+1}},
	\end{align*}
where $(a_1,...,a_d)$ is an $n$-tuple of positive integers and $a_{d+1}=a_1$.
\end{construction}

\begin{remark}\label{27}
	Given $d$-tuple of positive integers $(a_1,...,a_d)$, then $\bar{D}_{(a_1,...,a_d)}Q_{\alpha}$ means we do all the operations simultaneously instead of "composition of operations". For example, let $\alpha=(1 2 3)$ and $\bar{D}_{(1,2,3)}=\bar{D}_{12}\bar{D}_{23}\bar{D}_{31}$. If we do the operations simultaneously, the new quiver is
	\begin{align*}
		1 \rightarrow 3 \rightarrow 2 \rightarrow 1.
	\end{align*}
	But, if we do it as compositions, $\bar{D}_{31}Q_\alpha$ is
	\begin{align*}
		3 \rightarrow 2 ,\quad 2 \rightarrow 3,\quad  3 \rightarrow 1.
	\end{align*}
	This quiver has two arrows with source $3$. In this case, $\bar{D}_{23}$ cannot act on this quiver by Construction \ref{26}. This is the reason why we want to do all the operations simultaneously, otherwise, we don't know which arrow to replace.
\end{remark}

The new quiver $\bar{D}_{a_2 a_1} Q_\alpha$ may not be of the form $\Phi_n(\beta)$, i.e not correspond to a well defined element $\beta$ in the permutation group $S_n$ under the map $\Phi_n$. But, we have the following lemma.

\begin{lemma}\label{28}
	Let $\alpha \in S_n$ and $Q_\alpha = \Phi_n(\alpha)$ is the corresponding quiver. Given $d$ distinct vertexes (or positive integers) $a_1,...,a_d$, then $\bar{D}_{(a_1,...,a_d)}Q_{\alpha}$ corresponds to an element in $S_n$.
\end{lemma}

\begin{proof}
	In the construction, this procedure only changes the source of each arrow and fixes the target. Therefore, we pick $d$ arrows such that their sources are $a_1,...,a_d$ respectively. By the construction, substitute the source $a_i$ by $a_{i+1}$, where $i \leq d-1$ and $a_1$ by $a_d$, and get a new quiver $(\bar{D}_{(a_1,...,a_d)})Q_{\alpha}$. Clearly, this quiver still represents for an element in $S_n$, because each integer $k$ ($k \leq n$) appears once as a target and once as a source.	
\end{proof}

\begin{remark}\label{29}
	From the proof of the lemma, we have $Q_{\alpha'}=(\bar{D}_{(a_1,...,a_d)})Q_{\alpha}$, where $\alpha'=(a_1 \text{ }a_{2} \text{ }... \text{ }a_d) \alpha$.
\end{remark}

Now, consider the monomial $\beta(\Phi_n((12...n)))=X_{1 2}X_{2 3}...X_{n 1}$ which is a term in $tr(X^n)$. We use the permutation $(1 2 ...n)$ to represent this monomial or the quiver
\begin{align*}
	1 \rightarrow 2 \rightarrow ... \rightarrow n \rightarrow 1 \text{ .}
\end{align*}
We use $D_{21}$ (refer to Definition \ref{14}) acting on this term, then we get
\begin{align*}
	D_{21}X_{1 2}X_{2 3}...X_{n 1}=X_{2 2}X_{2 3}...X_{n 1}.
\end{align*}
The new term $X_{2 2}X_{2 3}...X_{n 1}$ can be represented by a quiver
\begin{align*}
	& 2 \rightarrow 2\\
	& 2 \rightarrow 3 \rightarrow ... \rightarrow n \rightarrow 1 \text{ .}
\end{align*}
In this way, if we use quivers to represent the monomials, then $D_{a_1 a_2}$ acting on monomials is the same as $\bar{D}_{a_1 a_2}$ acting on the corresponding quivers. Hence, if $D_{a_1 a_2}...D_{a_d a_1} X$ is a nonzero monomial, then it can be represented by a permutation by Remark \ref{27}. With the discussion above, we have the following lemma.

\begin{lemma}\label{210}
	Let $\alpha \in S_n$. $Q_\alpha$ is the corresponding quiver and $M_\alpha$ is the corresponding monomial. We have $\beta(\bar{D}Q_\alpha)=DM_\alpha$, where $D=D_{(a_1,...,a_d)}$ and $\bar{D}=\bar{D}_{(a_1,...,a_d)}$, where $(a_1,...,a_d)$ is an $d$-tuple of positive integers.
\end{lemma}

\begin{definition}\label{211}
	Let $X$ be a monomial in $\mathbb{C}[X_{ij}]_{i,j \geq 1}$ and let $D$ be a (formal) differential operator. If $DX \neq 0$, then we say $D$ is a non-trivial operator (with respect to $X$). In this section, we prefer to consider the differential operator $D=D_{(a_1,...,a_n)}$.
\end{definition}

\begin{definition}\label{212}
	Let $S=\{t_i, i \geq 1\}$ be a set of variables, define $\mathbb{M}_t$ is the set of all monomials with variables $X_{t_i t_j}, i,j \geq 1$. Given an infinite sequence of positive integers $a=(a_1,a_2,...)$, define the evaluation map $ev_a: \mathbb{M}_t \rightarrow \mathbb{M}$,
	\begin{align*}
		ev_a(X_{t_i t_j})=X_{a_i a_j}.
	\end{align*}
	If $M_t$ is a monomial in $\mathbb{M}_t$, we define $M_t(a_1,a_n...)=ev_a(M_t)$.
	
	Similar to Definition \ref{17}, we introduce the following notation,
	\begin{align*}
		& X_{(t_1,...,t_n)}=\left(\prod_{i=1}^{n-1}X_{t_i t_{i+1}}\right)X_{t_n t_1},\\
		& D_{(t_1,...,t_n)}=:\left(\prod_{i=1}^{n-1}D_{t_i t_{i+1}}\right)D_{t_n t_1}:.
	\end{align*}
	Finally, we define $\bar{W}([d])= \frac{1}{d} :Tr((D_{t_i t_j})_{i,j \geq 1})^d:$	
\end{definition}

\begin{theorem*}\emph{\textbf{\ref{112}}}
	For any $g \in \mathbb{C}S_n$,
	\begin{align*}
		\Phi(K_{1^{n-d}d}g)=W([d]) \Phi(g).
	\end{align*}
\end{theorem*}

\begin{proof}
Let $g \in S_n$ and we can write it in disjoint cycles
\begin{align*}
g=(c_1 \text{ }...\text{ }c_{\lambda_1})(c_{\lambda_1+1} \text{ }...\text{ }c_{\lambda_1+\lambda_2})\text{ }...\text{ }(c_{n-\lambda_m+1}\text{ }...\text{ }c_n),
\end{align*}
where $\lambda=(\lambda_1,...,\lambda_m)$ is the partition corresponding to $g$.

$W([d])$ is an infinite sum of operators $D_{(b_1,...,b_d)}$, $b_i$ are positive integers, (see Definition \ref{17}) and $\Phi(g)=\prod_{i=1}^{m}p_{\lambda_i}$ is an infinite sum of monomials in the form
	\begin{align*}
		M(a_1,...,a_n)=X_{(a_1,...,a_{\lambda_1})}...X_{(a_{n-\lambda_m+1},...,a_{n})}.
	\end{align*}
	Given any monomial $M(a_1,...,a_n)$, there are only finitely many operators $D_{(b_1,...,b_d)}$ in $W([d])$ such that $D_{(b_1,...,b_d)}M(a_1,...,a_n)\neq 0$. Hence, $W([d]) M(a_1,...,a_n)$ is a finite sum of monomials. To analyze these monomials, we first consider the generic case $M_t$. Then, we go back to $M(a_1,...,a_n)$ as the evaluation of $M_t$ at some $n$-tuple of integers,
	\begin{align*}
		M(a_1,...,a_n)=ev_a(M_t),
	\end{align*}
where $a=(a_1,...,a_n)$ is an $n$-tuple of positive integers.

We replace $W([d])$ by $\bar{W}([d])$ (see Definition \ref{212}) and $g$ by $\bar{g}$, where
\begin{align*}
\bar{g}=(t_1 \text{ }...\text{ }t_{\lambda_1})(t_{\lambda_1+1} \text{ }...\text{ }t_{\lambda_1+\lambda_2})\text{ }...\text{ }(t_{n-\lambda_m+1}\text{ }...\text{ }t_n).
\end{align*}
We consider a special case $M_t=X_{(t_1,...,t_{\lambda_1})}...X_{(t_{n-\lambda_m+1},...,t_{n})}$. In this case, we prefer to use the notation $M_{\bar{g}}$ for $M_t$. Now we will calculate $\bar{W}([d])M_{\bar{g}}$. By Remark \ref{29} and Lemma \ref{210}, let $i_1,...,i_d$ be distinct integers in $\{1,...,n\}$, we have
	\begin{align*}
		D_{(t_{i_1}, t_{i_2},...,t_{i_d})} M_{\bar{g}} = M_{\bar{\sigma} \bar{g}},
	\end{align*}
where $\bar{\sigma}$ is the $d$-cycle $(t_{i_1}\text{ }...\text{ }t_{i_d}) \in \bar{S}_n = Aut\{t_1,...,t_n\}$. Since $D_{(t_{i_1},...,t_{i_d})} M_{\bar{g}}$ is nonzero if and only if ${i_j} \in \{1,...,n\}, 1 \leq j \leq d$, so we have
	\begin{align*}
		\sum_{(i_1,...,i_d), \atop i_j \in \{1,...,n\} \text{ and distinct}}D_{(t_{i_1},..., t_{i_d})}M_{\bar{g}}=d \sum_{\bar{\sigma} \text{ } d \text{-cycle in }\bar{S}_{n}} M_{\bar{\sigma} \bar{g}}.
	\end{align*}
	Here we understand there are $d$ $d$-tuples $(i_1,...,i_d)$ giving rise to the same $d$-cycle. Hence, we have a coefficient $d$ at the right side of the above equation. Then, we have the following formula
	\begin{align*}
	\bar{W}([d]) M_{\bar{g}} &= \frac{1}{d} \sum_{(i_1,...,i_d), \atop i_j \in \{1,...,n\} \text{ and distinct}} D_{(t_{i_1},...,t_{i_d})}M_{\bar{g}}\\
		&= \frac{1}{d}\sum_{(i_1,...,i_d), \atop i_j \in \{1,...,n\} \text{ and distinct}} M_{(t_{i_1}...t_{i_d}) \bar{g}}\\ &= \sum_{\bar{\sigma} \text{ } d \text{-cycle in }\bar{S}_{n}} M_{\bar{\sigma} \bar{g}}.	
	\end{align*}

	Now we want to show for any $d$-tuple $(a_1,...,a_d)$ (with maybe some $a_i$ not distinct), we have
	\begin{equation}\label{eq31}
		W([d]) M_{\bar{g}}(a_1,...,a_n) =\sum_{\bar{\sigma} \text{ } d \text{-cycle in } \bar{S}_{n}} M_{\bar{\sigma} \bar{g}}(a_1,...,a_n).
	\end{equation}
We note that for any $n$-tuple $(a_1,...,a_n)$, the right hand side of \eqref{eq31} is always a sum of $\frac{1}{d}{n \choose d}d!$ monomials, each of which corresponds an unique element in $\bar{S}_n$. We hope for any $n$-tuple $(a_1,...,a_n)$, the left hand side is a sum of $\frac{1}{d}{n \choose d}d!$ monomials or we can find ${n \choose d}d!$ nontrivial operators in $W([d])$ with respect to $M_{\bar{g}}(a_1,...a_n)$. (Recall in the definition of $W([d])$, we have a coefficient $\frac{1}{d}$.) But the left hand side is complicated if the $a_{i}$ are not distinct. Indeed, if $a_{i}$ are not distinct, there are fewer nontrivial operators $D_{(a_{i_1},...,a_{i_d})}$ in $W([d])$ with respect to $M_{\bar{g}}(a_1,...a_n)$ than that in $\bar{W}([d])$ with respect to $M_{\bar{g}}$. For example, let's consider about the following case
	\begin{align*}
		M=X_{(t_1,t_2,t_3)}=X_{t_1 t_2}X_{t_2 t_3}X_{t_3 t_1}.
	\end{align*}
There are $6$ nontrivial differential operators $D_{(t_{i_1},t_{i_2},t_{i_3})}$ in $\bar{W}([3])$ with respect to $M$, where $(i_1,i_2,i_3)$ is any $3$-tuples such that $i_1,i_2,i_3 \in \{1,2,3\}$ and distinct. However, if we take $a_1 = a_2 =1,a_3=2$, we get only $3$ nontrivial operators in $W([3])$ with respect to $X_{(1,1,2)}$. They are $D_{(1,1,2)}$, $D_{(1,2,1)}$, $D_{(2,1,1)}$. In this case, we have to check whether we can get enough monomials on the left hand side of the equation.
	
	Before we discuss different cases, we first give some results which are based on basic calculations. The number of $d$-cycles in $S_n$ is $\frac{1}{d}{n \choose d}d!$. Given a monomial $M_{\bar{g}}$ of degree $n$, the number of non-trivial operators $D_{(t_{i_1},...,t_{i_d})}$ in $\bar{W}([d])$ corresponding to $M_{\bar{g}}$ is ${n \choose d}d!$. Each differential operator $D_{(t_{i_1},...,t_{i_d})}$ corresponds to a unique $d$-tuple $(t_{i_1},...,t_{i_d})$, which corresponds to a unique $d$-cycle $(t_{i_1}\text{ }...\text{ }t_{i_d})$. But, a $d$-cycle corresponds to $d$ $d$-tuples or $d$ differential operators in $\bar{W}([d])$.

Now we begin to prove Equation \eqref{eq31}.
		
	\textbf{Case 1,} $a_{i}$ are distinct, $1 \leq i \leq d$.
	
	In this case, each "non-trivial operator" $D_{(a_{i_1},..., a_{i_d})}$ corresponds to a unique $d$-cycle in $\bar{S}_n$. But this correspondence is not injective, it is an $d$ to $1$ correspondence. For example, when $d$ is $3$, we have
	\begin{align*}
		:D_{(a_1,a_2,a_3)}:=:D_{(a_2,a_3,a_1)}:=:D_{(a_3,a_1,a_2)}:.
	\end{align*}
	Hence, we get
	\begin{align*}
		W([d])M_{\bar{g}}(a_1,...a_n)=\sum_{\bar{\sigma} \text{ } d \text{-cycle in } S_{n}} M_{\bar{\sigma} \bar{g}}(a_,...,a_n).
	\end{align*}
	The number of non-trivial operators with respect to $X_{\bar{g}}(a_1,...a_n)$ in $W([d])$ is ${n \choose d}d!$ and each corresponds to a unique $d$-tuple in variables $t_i, 1 \leq i \leq n$.

	\textbf{Case 2,} $a_i$ are not all distinct and all $X_{a_i a_{i+1}}$ are distinct.
	
	First, we consider a special case that only two numbers of $\{ a_i \}_{1 \leq i \leq n}$ are the same and we assume that $a_p=a_q$. Now we consider the operator $D_{(a_{i_1},..., a_{i_d})}$.
	
	\begin{enumerate}
	\item If all $a_{i_j} \neq a_p$, then the non-trivial differential operator $D_{(a_{i_1},..., a_{i_d})}$ with respect to $X_{(a_1,...,a_n)}$ corresponds to a unique $d$-tuple in $t_i$. Under this condition, there are ${n-2 \choose d}d!$ $d$-tuples $(a_{i_1},...,a_{i_d})$ satisfying this condition and each corresponds to a unique $d$-tuple $(t_{i_1},...,t_{i_d})$.
	
	\item If only one number in the tuple $(a_{i_1},...,a_{i_d})$ is $a_p$ and we assume $a_{i_k}=a_p$, then the non-trivial differential operator $D_{(a_{i_1},..., a_{i_d})}$ corresponds to two $d$-tuples in variables $t_i, 1 \leq i \leq n$. Indeed, we have
	\begin{align*}
		 & D_{a_{i_{k-1}}a_{i_{k}}} X_{a_1 a_2}...X_{a_n a_1}
		= D_{a_{i_{k-1}}a_{p}} X_{a_1 a_2}...X_{a_n a_1}=\\
		=& \left(\sum_{c \geq 1}X_{a_{i_{k-1}} c} \frac{\partial}{\partial X_{a_{p} c}} \right)X_{a_1 a_2}...X_{a_n a_1}=\\
		=& \left(X_{a_{i_{k-1}} a_{p+1}} \frac{\partial}{\partial X_{a_p a_{p+1}}} + X_{a_{i_{k-1}} a_{q+1}} \frac{\partial}{\partial X_{a_p a_{q+1}}}\right)X_{a_1 a_2}...X_{a_n a_1}.
	\end{align*}
	The last equality holds because only these two terms in $D_{a_{i_{k-1}}a_p}$ act non-trivially on $X_{(a_1,...,a_n)}$ with our assumptions $a_p=a_q$. Compared with $(t_{i_1},...,t_{i_d})$, the differential operator $D_{(a_{i_1},..., a_{i_d})}$ now actually corresponds to two $d$-tuples in variables $t_i, 1 \leq i \leq n$. They are
	\begin{align*}
		& (t_{i_1},...,t_{i_{k-1}},t_{p},t_{i_{k+1}},...,t_{i_d}),\\ & (t_{i_1},...,t_{i_{k-1}},t_{q},t_{i_{k+1}},...,t_{i_d}).
	\end{align*}
Under this condition, there are $\frac{1}{2} {n-2 \choose d-1}{2 \choose 1}$ $d$-tuples $(a_{i_1},...,a_{i_d})$ satisfying the condition that only one number in the tuple $(a_{i_1},...,a_{i_d})$ is $a_p$, and each of them corresponds to two $d$-tuples in variables $t_i, 1 \leq i \leq n$.
	\item If there are two numbers in the tuple $(a_{i_1},...,a_{i_d})$ are $a_p$ and we assume they are $a_{i_l}=a_{i_k}=a_p$, then each non-trivial differential operator $D_{(a_{i_1},..., a_{i_d})}$ corresponds to two elements in the permutation group $\bar{S}_n$. Indeed, we have
\begin{align*}
		  & :D_{a_{i_{l-1}} a_{i_{l}}} D_{a_{i_{k-1}} a_{i_{k}}}:X_{a_1 a_2}...X_{a_n a_1}
		= :D_{a_{i_{l-1}} a_{p}} D_{a_{i_{k-1}} a_{p}}:X_{a_1 a_2}...X_{a_n a_1}.\\
\end{align*}
Since we only care about the non-trivial terms, so we have to calculate the differential operators $:D_{ a_{i_{l-1}}a_{p}} D_{a_{i_{k-1}}a_{p}}:$, of which the differential part is
\begin{align*}
\frac{\partial^2}{\partial X_{a_p a_{p+1}} \partial X_{a_q a_{q+1}}}.
\end{align*}
By definition, we know
\begin{align*}
		D_{a_{i_{l-1}} a_{p}}&= \sum_{c \geq 1}X_{a_{i_{l-1}} c} \frac{\partial}{\partial X_{a_{p} c}},\\
		D_{a_{i_{k-1}} a_{p}}&= \sum_{d \geq 1}X_{a_{i_{k-1}} d} \frac{\partial}{\partial X_{a_{p} d}}.
\end{align*}
	So, we have
	\begin{align*}
		& :D_{a_{i_{l-1}} a_{p}} D_{a_{i_{k-1}} a_{p}}:X_{a_1 a_2}...X_{a_n a_1}=\\
	   =& (\sum_{c,d \geq 1} X_{a_{i_{l-1}} c}X_{a_{i_{k-1}} d}\frac{\partial}{\partial X_{a_{p} c}}\frac{\partial}{\partial X_{a_{p} d}})X_{a_1 a_2}...X_{a_n a_1}=\\	
	   =& ( X_{a_{i_{l-1}} a_{p+1}}X_{a_{i_{k-1}} a_{q+1}}\frac{\partial}{\partial X_{a_{p} a_{p+1}}}\frac{\partial}{\partial X_{a_{p} a_{q+1}}}+\\
	   +& X_{a_{i_{l-1}} a_{q+1}}X_{a_{i_{k-1}} a_{p+1}}\frac{\partial}{\partial X_{a_{p} a_{q+1}}}\frac{\partial}{\partial X_{a_{p} a_{p+1}}} ) X_{a_1 a_2}...X_{a_n a_1}.
	\end{align*}
The last equality holds because all $X_{a_i a_j}$ are distinct by the assumption of \textbf{Case 2}. Hence, $a_{q+1} \neq a_{p+1}$. Compared with $(t_{i_1},...,t_{i_d})$, the differential operator $D_{(a_{i_1},..., a_{i_d})}$ corresponds to two $d$-tuples. They are
\begin{align*}
		& (t_{i_1},...,t_{i_k},...,t_{i_{l}},...,t_{i_d}),\\ & (t_{i_1},...,t_{i_{l}},...,t_{i_k},...,t_{i_d}).
\end{align*}
In this case, $D_{(a_{i_1},..., a_{i_d})}$ corresponds to two different $d$-tuples in variables $t_i, 1 \leq i \leq n$.	There are $\frac{1}{2}{n-2 \choose d-2}d!$ $d$-tuples $(a_{i_1},...,a_{i_d})$ satisfying the condition that there are two numbers in the tuple $(a_{i_1},...,a_{i_d})$ are $a_p$, and each of them corresponds to two $d$-tuples in variables $t_i, 1 \leq i \leq n$.
	\end{enumerate}
	
	Hence, in this special case, the number of $d$-tuples in variables $t_i, 1 \leq i \leq n,$ corresponding to the nontrivial differential operators with respect to the monomial $X_{(a_1,...,a_n)}$ is
	\begin{align*}
		{n-2 \choose d}d!+2\times \frac{1}{2}{n-2 \choose d-1}{2 \choose 1}d! + 2\times \frac{1}{2}{n-2 \choose d-2}d!={n \choose d}d!.
	\end{align*}
By the discussion above, the ${n \choose d}d!$ tuples are different. Recall ${n \choose d}d!$ is also the number of non-trivial operators $D_{(t_{i_1},...,t_{i_d})}$ in $\bar{W}([d])$ with respect to a fix monomial. So, we have
	\begin{align*}
		d \times W([d])M_{\bar{g}}(a_1,...,a_n)&= \sum_{(i_1,...,i_d), \atop i_j \in \{1,...,n\} \text{ and distinct}} \left(D_{(t_{i_1},...,t_{i_d})}M_{\bar{g}}\right)(a_1,...,a_n)\\
		&=\sum_{(i_1,...,i_d), \atop i_j \in \{1,...,n\} \text{ and distinct}} M_{(i_1...i_d) \bar{g}}(a_1,...,a_n)\\
&=d \times \sum_{\bar{\sigma} \text{ } d \text{-cycle in } \bar{S}_{n}} M_{\bar{\sigma} \bar{g}}(a_1,...,a_n)\\.
	\end{align*}
For the general case of $s$ integers $a_{j_1}=a_{j_2}=...=a_{j_2}$ but $X_{a_i a_{i+1}}$ all distinct, the same argument proves what we want. We leave it to the reader to check this.
		
	\textbf{Case 3,} $a_i$ are not all distinct, and some $X_{a_i a_{i+1}}$ are the same.
	
    We still consider a special case that only two terms in $X_{(a_1,...,a_n)}$ are the same. We assume $X_{a_{p} a_{p+1}}=X_{a_{q} a_{q+1}}$, where $p \neq q$ and $p+1,q+1$ means the addition mod $n$. Under this condition, we consider some examples. First, we have $a_p = a_q$ and $a_{p+1}=a_{q+1}$ and the other $a_i$ are distinct. Some examples are
    \begin{align*}
    	X_{11}X_{11}, p=1, q=2,\\
    	X_{12}X_{21}X_{12}X_{23}X_{31},p=1,q=3.
    \end{align*}
	These are cases we want to study.
	
	At the same time, there are some other examples.
	\begin{align*}
		X_{11}X_{11}X_{12}X_{21}.
	\end{align*}
	In this example, we have $X_{11}^2$ and another term $X_{12}$, which means there are some other $a_i$ such that $a_i=a_p$. To solve this type of monomials, it is a combination of \textbf{Case 2} and \textbf{Case 3}.
	
	Now, let's consider the problem that only two terms in $X_{(a_1,...,a_n)}$ are the same
\begin{align*}
X_{a_p a_{p+1}}=X_{a_q a_{q+1}}, \quad a_p = a_q, a_{p+1}=a_{q+1}, \quad p \neq q,
\end{align*}
and the other $a_i$ are distinct. In this case, we still discuss the nontrivial operators $D_{(a_{i_1},...,a_{i_d})}$.
	
	\begin{enumerate}
		\item If all $a_{i_j} \neq a_p$, then $D_{(a_{i_1},...,a_{i_d})}$ corresponds to a unique $d$-tuple in variables $t_i, 1\leq i \leq n$. Under this condition, although $a_{i_j} \neq a_p$, $a_{i_j}$ could be $a_{p+1}$. By our assumptions, we know that only two terms in $X_{(a_1,...,a_n)}$ are the same. Hence, there are ${n-2 \choose d}d!$ $d$-tuples $(a_{i_1},...,a_{i_d})$ satisfying this condition and each of them corresponds to a unique $d$-tuple in variables $t_i, 1\leq i \leq n$, by the conclusion of \textbf{Case 2}.
		
		\item Only one integer in $\{a_{i_j}\}_{1 \leq j \leq d}$ is $a_p$, say $a_{i_k}=a_p$.
		
		First, assume all $a_{i_j}$ are not $a_{p+1}$. Then, we have
		\begin{align*}
			& D_{a_{i_{k-1}}a_{i_{k}}} X_{(a_1,...,a_n)}
			= D_{a_{i_{k-1}}a_{p}} X_{a_1 a_2}...X_{a_n a_1}=\\
			=& (\sum_{c \geq 1}X_{a_{i_{k-1}} c} \frac{\partial}{\partial X_{a_{p} c}} )X_{a_1 a_2}...X_{a_n a_1}=\\
			=& (X_{a_{i_{k-1}} a_{p+1}} \frac{\partial}{\partial X_{a_p a_{p+1}}} )X_{a_1 a_2}...X_{a_n a_1}=\\
			=& (X_{a_{i_{k-1}} a_{p+1}} \frac{\partial}{\partial X_{a_p a_{p+1}}} ) X_{a_p a_{p+1}}^2 ... \text{ }.\\
		\end{align*}
		The last equality holds because we have $X_{a_{p} a_{p+1}}=X_{a_{q} a_{q+1}}$. We note there is a square $X_{a_p a_{p+1}}^2$ in the monomial $X_{(a_1,...,a_n)}$. Hence, we will get two (same) monomials. Compared with $(t_{i_1},...,t_{i_d})$, this differential operator $D_{(a_{i_1},...,a_{i_d})}$ corresponds to two $d$-tuples in variables $t_i, 1\leq i \leq n$. They are
		\begin{align*}
			 &(t_{i_1},...,t_{i_{k-1}},t_{p},t_{i_{k+1}},...,t_{i_d}),\\
			 &(t_{i_1},...,t_{i_{k-1}},t_{q},t_{i_{k+1}},...,t_{i_d}).
		\end{align*}
		Similarly, if some $a_{i_j}$ are $a_{p+1}$, then the conclusion follows by the combination of the above argument and the argument in \textbf{Case 2}. (If it contains both $a_p$ and $a_q$, then it corresponds to $4$ permutations.) We conclude all non-trivial differential operators $D_{(a_{i_1},...,a_{i_d})}$ in the case correspond to ${2 \choose 1}{n-2 \choose d-1}d!$ $d$-tuples in variables $t_i, 1\leq i \leq n$.
		
		\item Two of the integers $a_{i_j}, 1 \leq j \leq d$ are $a_p$ and we assume they are $a_{i_l}=a_{i_k}=a_p$.
		
		Similarly, assume all $a_{i_j}$ are not $a_{p+1}$. We have
		\begin{align*}
			& :D_{a_{i_{l-1}} a_{i_{l}}} D_{a_{i_{k-1}} a_{i_{k}}}:X_{a_1 a_2}...X_{a_n a_1}\\
			= & :D_{a_{i_{l-1}} a_{p}} D_{a_{i_{k-1}} a_{p}}:X_{a_1 a_2}...X_{a_n a_1}\\
			=& :D_{a_{i_{l-1}} a_{p}} D_{a_{i_{k-1}} a_{p}}:X_{a_p a_{p+1}}^2...\\
			=& (\sum_{c,d \geq 1} X_{a_{i_{l-1}} c}X_{a_{i_{k-1}} d}\frac{\partial}{\partial X_{a_{p} c}}\frac{\partial}{\partial X_{a_{p} d}})X_{a_p a_{p+1}}^2...\\	
			=& (X_{a_{i_{l-1}} a_{p+1}}X_{a_{i_{k-1}} a_{p+1}} \frac{\partial^2}{\partial^2 X_{a_{p} a_{p+1}}})X_{a_p a_{p+1}}^2...  .
		\end{align*}
		Note we have a square $X_{a_p a_{p+1}}^2$. Hence, we will get two (same) monomials. Compared with $(t_{i_1},...,t_{i_d})$, this differential operator $D_{(a_{i_1},...,a_{i_d})}$ corresponds to two $d$-tuples in variables $t_i, 1\leq i \leq n$. They are
		\begin{align*}
			&(t_{i_1},...,t_{i_k},...,t_{i_{l}},...,t_{i_d}),\\ &(t_{i_1},...,t_{i_{l}},...,t_{i_k},...,t_{i_d}).
		\end{align*}
		Similarly, if some $a_{i_j}$ are $a_{p+1}$, then the conclusion follows by the combination of the above argument and \textbf{Case 2}. (If it contains both $a_p$ and $a_q$, then it corresponds to $4$ permutations.) We conclude all non-trivial differential operators $D_{(a_{i_1},...,a_{i_d})}$ in the case correspond to ${n-2 \choose d-2}d!$ $d$-tuples in variables $t_i, 1\leq i \leq n$.
	\end{enumerate}
By the discussion above, the number of $d$-tuples in variables $t_i, 1 \leq i \leq n,$ corresponding to the nontrivial differential operators with respect to the monomial $X_{(a_1,...,a_n)}$ is
	\begin{align*}
		{n-2 \choose d}d!+2\times \frac{1}{2}{n-2 \choose d-1}{2 \choose 1}d! + 2\times \frac{1}{2}{n-2 \choose d-2}d!={n \choose d}d!.
	\end{align*}
These ${n \choose d}d!$ tuples are different. Hence, we have
	\begin{align*}
		d \times W([d])M_{\bar{g}}(a_1,...a_n)=d \times \sum_{\sigma \text{ } d \text{-cycle in } \bar{S}_{n}} M_{\sigma \bar{g}}(a_,...,a_n).
	\end{align*}
For the general case that there are $k$ same factors in $X_{a_1 a_2}...X_{a_n a_1}$, the same argument proves what we want. We leave it to the reader to check.

	Combining the above three cases, we get the following formula by summing over all monomials $M_{\bar{g}}(a_1,...,a_n)=X_{(a_1,...,a_n)}$ in $\Phi(g)$,
	\begin{align*}
		\Phi(K_{1^{n-d}d}g)=W([d]) \Phi(g).
	\end{align*}
\end{proof}

\section{Another Definition of $W([n])$}

In this section, we will consider $W([n])$ as a differential operator on the ring $\mathbb{C}[p_1,p_2,...]$ or $\mathbb{C}[[p_1,p_2,...]]$ by Theorem \ref{18}.

\subsection{Definition of $\Delta_n$}

Consider the cut-and-join operator $\Delta$ \cite{MR1249468},

\begin{equation}\label{eq40}
	\Delta= \frac{1}{2}\sum_{i \geq 1}\sum_{j\geq 1}(ijp_{i+j}\frac{\partial^2}{\partial p_i \partial p_j}+(i+j)p_i p_j \frac{\partial}{\partial p_{i+j}}).
\end{equation}

Recall Definition \ref{19} and \ref{110}. We have the following proposition.

\begin{proposition}\label{31}
	For any $g \in \mathbb{C}S_n$,
	\begin{align*}
		\Phi(K_{1^{n-2}2}g)=\Delta \Phi(g).
	\end{align*}
\end{proposition}

\begin{proof}
	Goulden proves this in \cite{MR1249468} Prop 3.1.
\end{proof}

\begin{definition}\label{32}
	For any permutation $\delta \in S_d$, let $\delta = \delta_1...\delta_m$, which is the decomposition of $\delta$ into disjoint cycles. For a positive integer $N \leq d$, $N \in \delta_i$ means $\delta_i (N) \neq N$. Fix $d$ positive integers $a_j$, where $1 \leq j \leq d$. Define $\hat{p}_\delta(a_1,...,a_d)$ to be the monomial
	\begin{align*}
		\hat{p}_\delta(a_1,...,a_d) = \prod_{i=1}^{m} p_{\sum_{j \in \delta_i}a_j}.
	\end{align*}
	Similarly, define $\frac{\partial}{\partial \hat{p}_\delta}(a_1,...,a_d)$ to be the operator on $\mathbb{C}[[p_1,p_2,...]]$,
	\begin{align*}
		\frac{\partial}{\partial \hat{p}_\delta}(a_1,...,a_d) = \prod_{i=1}^{m}((\sum_{j \in \delta_i}a_j) \frac{\partial}{\partial p_{\sum_{j \in \delta_i}a_j}}).
	\end{align*}
	If we fix positive integers $d$ and $a_1$,...,$a_d$, we abbreviate $\hat{p}_\delta(a_1,...,a_d)$ by $\hat{p}_\delta$ and $\frac{\partial}{\partial \hat{p}_\delta}(a_1,...,a_d)$ by $\frac{\partial}{\partial \hat{p}_\delta}$.
\end{definition}

\begin{remark}\label{33}
		For any element $\delta \in S_d$, it can be written as the product of disjoint cycles. In this paper, we also write "$1$-cycle" explicitly to make the above notations clearer. For example, let's consider the permutation $(123) \in S_4$. In this paper we prefer to write it as $(123)(4)$. In particular, we define an integer $n$ contained in a $1$-cycle $(n')$ if and only if $n=n'$.
\end{remark}

\begin{example}\label{34}
	Let $\delta=(123)(4) \in S_4$, then we have
	\begin{align*}
		 \hat{p}_\delta(a_1,...,a_4)&=p_{a_1+a_2+a_3}p_{a_4},\\
		 \frac{\partial}{\partial \hat{p}_\delta}(a_1,...,a_4)&=(a_1+a_2+a_3)a_4 \frac{\partial^2}{\partial p_{a_1+a_2+a_3} \partial p_{a_4}}.
	\end{align*}
\end{example}

\begin{remark}\label{35}
	Given $\delta\in S_d$, we consider $\hat{p}_\delta$ as a map from $\mathbb{Z}^{d}_{> 0}$ to $\mathbb{C}[p_1,p_2,...]$ and $\frac{\partial}{\partial \hat{p}_\delta}$ as a map from $\mathbb{Z}^{d}_{> 0}$ to $\mathbb{C}[\frac{\partial}{\partial p_1},\frac{\partial}{\partial p_2},...]$. Generally, we can introduce variables $t_i$ and we write $\hat{p}_\delta$ and $\frac{\partial}{\partial \hat{p}_\delta}$ in the following form similar to Definition \ref{212},
	\begin{align*}
		 \hat{p}_\delta(t_1,...,t_d) &= \prod_{i=1}^{m} p_{\sum_{j \in \delta_i}t_j},\\
		 \frac{\partial}{\partial \hat{p}_\delta}(t_1,...,t_d) &= \prod_{i=1}^{m}((\sum_{j \in \delta_i}t_j) \frac{\partial}{\partial p_{\sum_{j \in \delta_i}t_j}}).
	\end{align*}
\end{remark}

\begin{definition}\label{36}
	Consider the $d$-cycle $(d \text{ }...\text{ } 2 \text{ } 1)$ in $S_d$. We define the bijective map $\phi_d$ of $S_d$ as
\begin{align*}
\phi_d(\delta)=(d\text{ }...\text{ }1)\delta, \quad \delta \in S_d.
\end{align*}
If we fix $d$, we will use $\phi$ to represent this map.
\end{definition}

\begin{definition}\label{37}
	We define the differential operator $\Delta_d$ on the polynomial ring $\mathbb{C}[p_1,p_2,...]$ as
	\begin{align*}
		\Delta_d =\frac{1}{d} \sum_{\delta \in S_d}\sum_{a_1,...,a_d \geq 1} \hat{p}_{\phi(\delta)}(a_1,...,a_d) \frac{\partial}{\partial \hat{p}_\delta}(a_1,...,a_d).
	\end{align*}
\end{definition}

\begin{remark}
	The definition of the operator $\Delta_d$ depends on the map $\phi_d(\delta) = (d \text{ }...\text{ }1) \delta$, where $(d \text{ }...\text{ }1)$ is a $d$-cycle. Indeed, we can replace $(d...1)$ by any $d$-cycle in $S_d$ and define a new bijective map of $S_d$, which will give the same operator $\Delta_d$. We will prove this property in Corollary \ref{3033} and \ref{317}.
\end{remark}

Now we give two examples about the operator $\Delta_d$.
\begin{example}\label{38}
	\begin{align*}
		\Delta_2= \frac{1}{2}\sum_{i\geq1}\sum_{j\geq1}(ijp_{i+j}\frac{\partial^2}{\partial p_i \partial p_j}+(i+j)p_i p_j \frac{\partial}{\partial p_{i+j}}),
	\end{align*}
	where the first part corresponds to $(1)(2) \in S_2$ and the second part corresponds to $(1 2) \in S_2$. We see that $\Delta_2$ is the cut-and-join operator $\Delta$ \eqref{eq40}.
	\begin{align*}
		\Delta_3 = \frac{1}{3}\sum_{i_1,i_2,i_3 \geq 1}
		(&i_1i_2i_3 p_{i_1+i_2+i_3}\frac{\partial^3}{\partial p_{i_1} \partial p_{i_2}\partial p_{i_3}}+  & (1)(2)(3)\\
		+&i_1(i_2+i_3)p_{i_1+i_3}p_{i_2}\frac{\partial^2}{\partial p_{i_1} \partial p_{i_2+i_3}}+ & (1)(23)\\
		+&i_2(i_1+i_3)p_{i_1+i_2}p_{i_3}\frac{\partial^2}{\partial p_{i_2} \partial p_{i_1+i_3}}+ & (2)(13)\\
		+&i_3(i_1+i_2)p_{i_3+i_2}p_{i_1}\frac{\partial^2}{\partial p_{i_3} \partial p_{i_1+i_2}}+ & (3)(12)\\
		+&(i_1+i_2+i_3)p_{i_1}p_{i_2}p_{i_3} \frac{\partial}{\partial p_{i_1+i_2+i_3}}+ & (123)\\
		+&(i_1+i_2+i_3)p_{i_1+i_2+i_3}\frac{\partial}{\partial p_{i_1+i_2+i_3}}) & (132).
	\end{align*}
	where each summation corresponds to the permutation $(1)(2)(3)$,$(1)(23)$,$(2)(13)$,
	$(3)(12)$,$(123)$,$(321)$ in turn.
\end{example}

\begin{definition}\label{39}
	Let $n$ and $d$ be positive integers, $d \leq n$. $C_{n,d}$ is the set of all $d$-cycles in $S_n$ and $\bar{C}_{n,d}$ is the set of all $d$-tuples $[a_1,...,a_d]$ with positive integers $a_i$ such that $1 \leq a_i \leq n$ and $a_i \neq a_j$ if $i \neq j$. We define a map $\pi_{n,d}:\bar{C}_{n,d} \rightarrow C_{n,d}$ such that
\begin{align*}
\pi_{n,d}([a_1,...,a_d])=(a_1 \text{ }...\text{ }a_d).
\end{align*}
Clearly, this map is $d$-to-$1$. Given an $d$-tuple $\bar{\sigma} \in \bar{C}_{n,d}$ and a permutation $g \in S_n$, we define the action of $\bar{C}_{n,d}$ on $S_n$ as following,
	\begin{align*}
		\bar{\sigma}g:=\pi_{n,d}(\bar{\sigma})g.
	\end{align*}
We define $\mathbb{C} \bar{C}_{n,d}= \oplus_{[a_1,...,a_d]\in \bar{C}_{n,d}} \mathbb{C}[a_1,...,a_d]$ as the vector space with basis the elements of $\bar{C}_{n,d}$. Finally, we define the element $\bar{K}_{1^{n-d}d} \in \mathbb{C} \bar{C}_{n,d}$ as the sum of all $d$-tuples in $\bar{C}_{n,d}$. 	
\end{definition}
In this paper, given positive integers $n$ and $d$, we abbreviate $\pi_{n,d}$ by $\pi$ and consider $\pi$ as a linear map from $\mathbb{C} \bar{C}_{n.d}$ to $\mathbb{C}C_{n,d}$. We are going to use $\bar{K}_{1^{n-d}d}$ to show that
\begin{align*}
d\Phi(K_{1^{n-d}d}g)=\Phi(\bar{K}_{1^{n-d}d}g)= d\Delta_d \Phi(g).
\end{align*}

\subsection{Proof when $d=3$}
Given $\bar{\sigma} \in \bar{C}_{n,3}$ and $g \in S_n$, we will calculate $\bar{\sigma}g$ and translate it into differential operators and polynomials.
	
\begin{construction}\label{3009}
Let $\bar{\sigma}=[j_3,j_2,j_1]$ be a $3$-tuple. We are going to classify elements $g \in S_n$ according to the occurrence of $j_1,j_2,j_3$ in the disjoint cycles appearing in $g$. There are $6$ cases with respect to $\bar{\sigma}$, one for each permutation of $S_3$,
\begin{enumerate}
		\item $g=(j_1...)(j_2...)(j_3...)...\text{ },$
		\item $g=(j_1...)(j_2...j_3...)...\text{ },$
		\item $g=(j_1...j_3...)(j_2...)...\text{ },$
		\item $g=(j_1...j_2...)(j_3...)...\text{ },$
		\item $g=(j_1...j_2...j_3...)...\text{ },$
		\item $g=(j_1...j_3...j_2...)...\text{ }.$
\end{enumerate}

Clearly, for any element $g \in S_n$, it falls into one and only one case with respect to $\bar{\sigma}$. Now, consider case (4) $g=(j_1\colorbox{red}{...}j_2\colorbox{blue}{...})(j_3\colorbox{green}{...})...$, where the red dots represent the digits after $j_1$ before $j_2$, the blue dots represent the other digits after $j_2$ before $j_1$ (since it is a cycle, so the last element will go back to $j_1$) and the green dots represent the other digits in the cycle of $j_3$.
We use the following steps to calculate $\bar{\sigma}g$:

\begin{enumerate}
	\item Restrict $g=(j_1\colorbox{red}{...}j_2\colorbox{blue}{...})(j_3\colorbox{green}{...})...$ to the element $(j_1 j_2)(j_3)$ by forgetting all digits except $j_1$,$j_2$,$j_3$ but preserving the cycle structure. $(j_1 j_2)(j_3)$ can be considered as an element in $Aut\{j_1,j_2,j_3\}$. Let $g_{\bar{\sigma}}=(j_1 j_2)(j_3)$.
	\item  Calculate $[j_3,j_2,j_1]g_{\bar{\sigma}}=(j_1)(j_2 j_3)$.
	\item  Insert all numbers forgotten by the restriction into $\bar{\sigma}g_{\bar{\sigma}}$, then we have the consequence,
	\begin{align*}
		\bar{\sigma}g=(j_1\colorbox{red}{...})(j_2\colorbox{green}{...}j_3\colorbox{blue}{...})... \text{ .}
	\end{align*}
\end{enumerate}
This procedure works for all cases.
\end{construction}

\begin{remark} \label{310}
	\begin{itemize}
		\item Let $\bar{\sigma}=[3,2,1]$ and $\bar{\sigma'}=[1,3,2]$. Although $\pi(\bar{\sigma})=\pi(\bar{\sigma'})=(1 3 2)$, $g_{\bar{\sigma}}$ and $g_{\bar{\sigma'}}$ are not in the same type in general. For instance, assume $g=(12)(3)$. Consider $\bar{\sigma}=[3,2,1]$, so that hence $g_{\bar{\sigma}}=(j_1 j_2)(j_3)$, which is in Case (4). Now, consider $\bar{\sigma'}=[1 3 2]$, so that $g_{\bar{\sigma'}}=(j_3 j_1)(j_2)$, which is in Case (3).
		
		\item Given different $\bar{\sigma}_1$, $\bar{\sigma}_2$, we can get $g_{\bar{\sigma}_1}=g_{\bar{\sigma}_2}$. For example, if $g=(3 2 1)$,$\bar{\sigma}_1=[3,2,1]$ and $\bar{\sigma}_2=[1,3,2]$, then we have $g_{\bar{\sigma}_1}=(3 2 1)=(2 1 3)=g_{\bar{\sigma}_2}$.
\end{itemize}
\end{remark}

\begin{remark}\label{30008}
Let $g$ be a permutation in $S_n$, $n\geq 3$. We consider two $3$-tuples $\bar{\sigma}=[1,2,3]$ and $\bar{\sigma}'=[j_3,j_2,j_1]$, $j_1,j_2,j_3 \leq n$. Clearly, $g_{\bar{\sigma}'} \in Aut\{j_3,j_2,j_1\}$ and $g_{\bar{\sigma}} \in Aut\{1,2,3\}$. But, we want to compare the two permutations in the same permutation group $S_3=Aut\{1,2,3\}$. Hence, we have to fix a bijective map between $\{1,2,3\}$ and $\{j_3,j_2,j_1\}$. We construct the map by sending the largest integer in $\{j_3,j_2,j_1\}$ to $3$, smallest one to $1$ and the last one to $2$. This map will induce an isomorphism $\o: Aut\{j_3,j_2,j_1\} \rightarrow Aut\{1,2,3\}$. Hence, by an abuse of notations, $g_{\bar{\sigma}'} \in S_3$ means $\o (g_{\bar{\sigma}'}) \in S_3$.
\end{remark}

\begin{definition}\label{30009}
Let $\alpha$ be a permutation in $S_n$ and let $\bar{\sigma} \in \bar{C}_{n,3}$. We say $(\alpha,\bar{\sigma})$ is of type $i$, if $\alpha$ and $\bar{\sigma}$ corresponds to Case (i) in Construction \ref{3009}, $1 \leq i \leq 6$.
\end{definition}

Let $\omega=(j_d \text{ } ... \text{ } j_2 \text{ } j_1)$ be a $d$-cycle in $S_n$ (or a $d$-tuple $[j_d , ... ,j_2,j_1]$) and $\alpha=\alpha_1...\alpha_l$ be any permutation in $S_n$, where $\alpha_1...\alpha_l$ is the unique product of disjoint cycles. The following set $\mathcal{L}_i$ for fixed integer $i$, $1 \leq i \leq d$,
\begin{align*}
\mathcal{L}_i=\{l \mid \alpha^{l}(j_i) \text{ is any }j_{k}, 1 \leq k \leq d, l \geq 1\},
\end{align*}
is nonempty, because $\alpha^{n!}$ is the identity map on the set $\{1,...,n\}$, so $\alpha^{n!}(j_i)=j_i$ implies that $n!$ is contained in this set.

\begin{definition}\label{3010}
We define the "distance" between $j_i$ and the set $\{j_1,...,j_d\}$ with respect to the permutation $\alpha$ as
\begin{align*}
dist(j_i,\alpha,j_1,j_2,...,j_d) = min (\mathcal{L}_i).
\end{align*}
\end{definition}

\begin{example}\label{3011}
We give some examples about the definition above. Consider Case (5) in Construction \ref{3009},
\begin{align*}
\omega=(j_3 \text{ } j_2 \text{ } j_1), \quad \alpha=(j_1...j_2...j_3...)\alpha_2...\alpha_l,
\end{align*}
where $\alpha_1=(j_1...j_2...j_3...)$. $dist(j_3,\alpha,j_1,j_2,j_3)$ is the "distance" between $j_3$ and $j_1$ in the cycle $\alpha_1$, because $j_1$ is the first element in $\{j_1.j_2,j_3\}$ after $j_3$ under the action of $\alpha$. Similarly, $dist(j_2,\alpha,j_1,j_2,j_3)$ is the "distance" between $j_2$ and $j_3$. Clearly, $\sum_{1 \leq i \leq 3}dist(j_i,\alpha,j_1,j_2,j_3)$ is the length of the cycle $\alpha_1$.

Now, let's consider Case (1) in Construction \ref{3009},
\begin{align*}
\alpha=(j_1...)(j_2...)(j_3...)\alpha_4...\alpha_l.
\end{align*}
In this case, $dist(j_i,\alpha,j_1,j_2,j_3)$ is the length of the disjoint cycle containing $j_i$.
\end{example}

\begin{remark}\label{3012}
$\alpha,\omega$ are permutations in $S_n$, where $\omega$ is a $d$-cycle $(j_d \text{ }...\text{ } j_1)$. Let $\alpha'=\omega \alpha$. Then, we have
\begin{align*}
dist(j_i,\alpha,j_1,...,j_d)=dist(j_i,\alpha',j_1,...,j_d), \quad 1 \leq i \leq d.
\end{align*}
This property comes from the calculation in Construction \ref{3009}.
\end{remark}

\begin{definition}\label{311}
Given any permutation $\alpha \in S_n$, we define the map
\begin{align*}
& I_{\alpha,n,3}: \bar{C}_{n,3} \rightarrow \mathbb{Z}^{3}_{> 0}, \\
& I_{\alpha,n,3}([j_3,j_2,j_1])=(i_3,i_2,i_1),
\end{align*}
where $i_k=dist(j_k,\alpha,j_1,j_2,j_3), 1 \leq k \leq 3$.
\end{definition}

\begin{definition}\label{3013}
Let $\alpha$ be a permutation in $S_n$ and let $i_k$ be positive integers, $1 \leq k \leq 3$. $m$ is a positive integer such that $1 \leq m \leq 6$. Define the subset $\bar{C}^m_{n,3}(\alpha,i_3,i_2,i_1)$ of $\bar{C}_{n,3}$ as
\begin{align*}
\bar{C}^m_{n,3}(\alpha,i_3,i_2,i_1)=\{[j_3,j_2,j_1] \mid  i_k=dist(j_k,\alpha,j_1,j_2,j_3), 1 \leq k \leq 3, \\ (\alpha,[j_3,j_2,j_1]) \text{ is of type } m \}.
\end{align*}
\end{definition}

\begin{remark}\label{3014}
Let $\alpha$ be a permutation in $S_n$. We have
\begin{align*}
\bar{C}_{n,3}=\bigcup_{m=1}^6 \bigcup_{i_1,i_2,i_3 \geq 1} \bar{C}^m_{n,3}(\alpha,i_3,i_2,i_1).
\end{align*}
Given any $3$-tuple $[j_3,j_2,j_1]$, the "distance" $dist(j_i,\alpha,j_1,...,j_3)$ and the type of $(\alpha,[j_3,j_2,j_1])$ are uniquely determined. Hence, $\bigcup_{m=1}^6 \bigcup_{i_1,i_2,i_3 \geq 1} \bar{C}^m_{n,3}(\alpha,i_3,i_2,i_1)$ is a disjoint union. Also, there are only finitely many nonempty sets $\bar{C}_{n,3}(\alpha,i_3,i_2,i_1)$ in the above union.
\end{remark}

\begin{lemma}\label{3015}
Let $\alpha$ be a permutation in $S_n$ and let $i_1,i_2,i_3$ be three positive integers. We have the following formula
\begin{align*}
&\Phi(\sum_{[j_3,j_2,j_1] \in \bar{C}^1_{n,3}(\alpha,i_3,i_2,i_1)} [j_3,j_2,j_1] \alpha) =i_1i_2i_3 p_{i_1+i_2+i_3}\frac{\partial^3\Phi(\alpha)}{\partial p_{i_1} \partial p_{i_2}\partial p_{i_3}}\\
&=\hat{p}_{\phi((1)(2)(3))}(i_1,i_2,i_3) \frac{\partial}{\partial \hat{p}_{(1)(2)(3)}}(i_1,i_2,i_3)\Phi(\alpha),
\end{align*}
\begin{align*}
& \Phi(\sum_{[j_3,j_2,j_1] \in \bar{C}^2_{n,3}(\alpha,i_3,i_2,i_1)} [j_3,j_2,j_1] \alpha)=i_1(i_2+i_3)p_{i_1+i_3}p_{i_2}\frac{\partial^2\Phi(\alpha)}{\partial p_{i_1} \partial p_{i_2+i_3}}\\
&=\hat{p}_{\phi((1)(2 \text{ }3))}(i_1,i_2,i_3) \frac{\partial}{\partial \hat{p}_{(1)(2 \text{ }3)}}(i_1,i_2,i_3)\Phi(\alpha),
\end{align*}
\begin{align*}
& \Phi(\sum_{[j_3,j_2,j_1] \in \bar{C}^3_{n,3}(\alpha,i_3,i_2,i_1)} [j_3,j_2,j_1] \alpha)=i_2(i_1+i_3)p_{i_1+i_2}p_{i_3}\frac{\partial^2\Phi(\alpha)}{\partial p_{i_2} \partial p_{i_1+i_3}}\\
&=\hat{p}_{\phi((2)(1 \text{ }3))}(i_1,i_2,i_3) \frac{\partial}{\partial \hat{p}_{(2)(1 \text{ }3)}}(i_1,i_2,i_3)\Phi(\alpha),
\end{align*}
\begin{align*}
& \Phi(\sum_{[j_3,j_2,j_1] \in \bar{C}^4_{n,3}(\alpha,i_3,i_2,i_1)} [j_3,j_2,j_1] \alpha)=i_3(i_1+i_2)p_{i_3+i_2}p_{i_1}\frac{\partial^2\Phi(\alpha)}{\partial p_{i_3} \partial p_{i_1+i_2}}\\
&=\hat{p}_{\phi((3)(1 \text{ }2))}(i_1,i_2,i_3) \frac{\partial}{\partial \hat{p}_{(3)(1 \text{ }2)}}(i_1,i_2,i_3)\Phi(\alpha),
\end{align*}
\begin{align*}
& \Phi(\sum_{[j_3,j_2,j_1] \in \bar{C}^5_{n,3}(\alpha,i_3,i_2,i_1)} [j_3,j_2,j_1] \alpha)=(i_1+i_2+i_3)p_{i_1}p_{i_2}p_{i_3} \frac{\partial\Phi(\alpha)}{\partial p_{i_1+i_2+i_3}}\\
&=\hat{p}_{\phi((1\text{ }2 \text{ }3))}(i_1,i_2,i_3) \frac{\partial}{\partial \hat{p}_{(1 \text{ }2 \text{ }3)}}(i_1,i_2,i_3)\Phi(\alpha),
\end{align*}
\begin{align*}
&\Phi(\sum_{[j_3,j_2,j_1] \in \bar{C}^6_{n,3}(\alpha,i_3,i_2,i_1)} [j_3,j_2,j_1] \alpha)=(i_1+i_2+i_3)p_{i_1+i_2+i_3}\frac{\partial\Phi(\alpha)}{\partial p_{i_1+i_2+i_3}})\\
&=\hat{p}_{\phi((3 \text{ }2 \text{ }1))}(i_1,i_2,i_3) \frac{\partial}{\partial \hat{p}_{(3\text{ }2 \text{ }1)}}(i_1,i_2,i_3)\Phi(\alpha).
\end{align*}
\end{lemma}
We only give the proof of the first formula
\begin{align*}
\Phi(\sum_{[j_3,j_2,j_1] \in \bar{C}^1_{n,3}(\alpha,i_3,i_2,i_1)} [j_3,j_2,j_1] \alpha)=i_1i_2i_3 p_{i_1+i_2+i_3}\frac{\partial^3\Phi(\alpha)}{\partial p_{i_1} \partial p_{i_2}\partial p_{i_3}}.
\end{align*}
The other formulas can be proved similarly. Before we give the proof, we first prove some lemmas.

\begin{lemma}\label{30001}
Let $\alpha$ be a permutation in $S_n$ and let $i_1,i_2,i_3$ be three positive integers. $c_v$ is the number of disjoint cycles with length $i_v$ in $\alpha$. The number of elements in $\bar{C}^1_{n,3}(\alpha,i_3,i_2,i_1)$ is $\prod^3_{v=1}c_v i_v$.
\end{lemma}

\begin{proof}
If there is no disjoint cycles with length $i_v$ of $\alpha$ for some $1\leq v \leq 3$, then $\bar{C}^1_{n,3}(\alpha,i_3,i_2,i_1)$ is empty. Also, since $c_v=0$, we have $\prod^3_{v=1}c_v i_v=0$. The statement is true in this special case.

Now we assume that there is at least one disjoint cycle with length $c_v$ in $\alpha$. We first pick disjoint cycle $\alpha'_v$ with length $i_v$ in $\alpha$, $1 \leq v \leq 3$. The number of the choices of $\alpha'_1,\alpha'_2,\alpha'_3$ is $\prod^3_{v=1}c_v$. After we pick three disjoint cycles $\alpha'_1,\alpha'_2,\alpha'_3$, we can pick any integer $j_v$ from $\alpha'_v$, $1 \leq v \leq 3$, and these three integers form a unique $3$-tuple $[j_3,j_2,j_1]$ in $\bar{C}^1_{n,3}(\alpha,i_3,i_2,i_1)$. We can construct $i_1 i_2 i_3$ many $3$-tuples in $\bar{C}^1_{n,3}(\alpha,i_3,i_2,i_1)$ from these three disjoint cycles $\alpha'_1,\alpha'_2,\alpha'_3$. In this way, we can construct $\prod^3_{v=1}c_v i_v$ many $3$-tuples in $\bar{C}^1_{n,3}(\alpha,i_3,i_2,i_1)$. It is easy to prove that they are all elements in $\bar{C}^1_{n,3}(\alpha,i_3,i_2,i_1)$.
\end{proof}

\begin{remark}\label{30002}
Let $\alpha$ be a permutation in $S_n$ and let $\bar{\sigma}$ be an element in the set $\bar{C}^1_{n,3}(\alpha,i_3,i_2,i_1)$. We use the same notations for $\alpha$ and $\bar{\sigma}$ as in Lemma \ref{3009}, i.e.
\begin{align*}
\alpha=(j_1...)(j_2...)(j_3...)\alpha_4...\alpha_l, \quad \bar{\sigma}=[j_3,j_2,j_1].
\end{align*}
Also, by definition we have
\begin{align*}
    i_k=dist(j_k,\alpha,j_1,j_2,j_3), \quad 1 \leq k \leq 3.
\end{align*}
We assume the lengths of disjoint cycles $\alpha_v$, $4 \leq v \leq l$, are not $i_1,i_2,i_3$. By simple calculations, we have
\begin{align*}
    & \alpha=(j_1...)(j_2...)(j_3...)\rho_4...\rho_l & \rightarrow & \quad \bar{\sigma}\alpha=(j_3...j_2...j_1...)\rho_4...\rho_l\\
    & \Phi(\alpha)=p_{i_1} p_{i_2} p_{i_3}\Phi(\rho_4...\rho_l)& \rightarrow & \quad \Phi(\bar{\sigma}\alpha)=p_{i_1+i_2+i_3}\Phi(\rho_4...\rho_l) \text{ ,}
\end{align*}
and
\begin{align*}
    p_{i_1+i_2+i_3}\frac{\partial^3}{\partial p_{i_1}\partial p_{i_2}\partial p_{i_3}} \Phi(\alpha) = \Phi(\bar{\sigma}\alpha).
\end{align*}
Clearly, for any element $\bar{\sigma}'$ in $\bar{C}^1_{n,3}(\alpha,i_3,i_2,i_1)$, we have
\begin{align*}
\Phi(\bar{\sigma}'\alpha)=\Phi(\bar{\sigma}\alpha),
\end{align*}
which means
\begin{align*}
p_{i_1+i_2+i_3}\frac{\partial^3}{\partial p_{i_1}\partial p_{i_2}\partial p_{i_3}}\Phi(\alpha) = \Phi(\bar{\sigma}\alpha)=\Phi(\bar{\sigma}'\alpha).
\end{align*}
\end{remark}

Now we give the proof of Lemma \ref{3015}.
\begin{proof}
If $\bar{C}^1_{n,3}(\alpha,i_3,i_2,i_1)$ is empty, we assume that there is no disjoint cycle with length $i_1$ in $\alpha$. We have
\begin{align*}
\Phi(\sum_{[j_3,j_2,j_1] \in \bar{C}^1_{n,3}(\alpha,i_3,i_2,i_1)} [j_3,j_2,j_1] \alpha)=0.
\end{align*}
Since there is no disjoint cycle with length $i_1$ in $\alpha$, we have
\begin{align*}
\frac{\partial \Phi(\alpha)}{\partial p_{i_1}}=0.
\end{align*}
Hence, the equation holds
\begin{align*}
0=\Phi(\sum_{[j_3,j_2,j_1] \in \bar{C}^1_{n,3}(\alpha,i_3,i_2,i_1)} [j_3,j_2,j_1] \alpha)=i_1i_2i_3 p_{i_1+i_2+i_3}\frac{\partial^3\Phi(\alpha)}{\partial p_{i_1} \partial p_{i_2}\partial p_{i_3}}=0.
\end{align*}

Now we assume there is at least one disjoint cycle with length $i_v$ in $\alpha$. The number of disjoint cycles with length $i_v$ in $\alpha$ is $c_v$. By Lemma \ref{30001}, we know the number of elements in $\bar{C}^1_{n,3}(\alpha,i_3,i_2,i_1)$ is $\prod^3_{v=1}c_v i_v$. By Lemma \ref{30001} and Remark \ref{30002}, we have
\begin{align*}
( \sum_{[j_3,j_2,j_1] \in \bar{C}^1_{n,3}(\alpha,i_3,i_2,i_1)} [j_3,j_2,j_1] \alpha )=(\prod^3_{v=1}c_v i_v) \Phi(\alpha'),
\end{align*}
where $\alpha'=[j_3,j_2,j_1]\alpha$ for some $[j_3,j_2,j_1] \in \bar{C}^1_{n,3}(\alpha,i_3,i_2,i_1)$. By Remark \ref{30002}, we know $\Phi([j_3,j_2,j_1]\alpha)$ does not depend on the choice of $[j_3,j_2,j_1]$ in $\bar{C}^1_{n,3}(\alpha,i_3,i_2,i_1)$. By assumption, there are $c_v$ disjoint cycles with length $i_v$ in $\alpha$, it means that the order of $p_{i_v}$ in the monomial $\Phi(\alpha)$ is $c_v$. So, when we calculate $\frac{\partial \Phi(\sigma_v)}{\partial p_{i_v}}$, we will have a coefficient $c_v$, i.e.
\begin{align*}
p_{i_1+i_2+i_3} \frac{\partial^3 \Phi(\alpha)}{\partial p_{i_1}\partial p_{i_2}\partial p_{i_3}}=(\prod^3_{v=1}c_v) \Phi(\alpha').
\end{align*}
So, we have
\begin{align*}
\Phi(\sum_{[j_3,j_2,j_1] \in \bar{C}^1_{n,3}(\alpha,i_3,i_2,i_1)} [j_3,j_2,j_1] \alpha)=i_1i_2i_3 p_{i_1+i_2+i_3}\frac{\partial^3\Phi(\alpha)}{\partial p_{i_1} \partial p_{i_2}\partial p_{i_3}}.
\end{align*}
\end{proof}
	
Now we are ready to prove the theorem.
\begin{theorem}\label{313}
Let $g$ be an element in $\mathbb{C}S_n$. We have
	\begin{align*}
		3 \Phi(K_{31^{n-3}}g)=\Phi(\bar{K}_{31^{n-3}}g)=3\Delta_3 \Phi(g).
	\end{align*}
\end{theorem}

\begin{proof}
We assume that $g$ is a permutation in $S_n$, i.e. $g=\alpha \in S_n$. By Remark \ref{3014}, we have
\begin{align*}
\bar{C}_{n,3}=\bigcup_{m=1}^6 \bigcup_{i_1,i_2,i_3 \geq 1} \bar{C}^m_{n,3}(\alpha,i_3,i_2,i_1).
\end{align*}
Then, we get
\begin{align*}
\Phi(\bar{K}_{31^{n-3}}g)& =\Phi(\sum_{m=1}^6 \sum_{i_1,i_2,i_3 \geq 1}\sum_{[j_3,j_2,j_1] \in \bar{C}^i_{n,3}(\alpha,i_3,i_2,i_1)} [j_3,j_2,j_1] \alpha)\\
& = \sum_{i_1,i_2,i_3 \geq 1}
          (i_1i_2i_3 p_{i_1+i_2+i_3}\frac{\partial^3}{\partial p_{i_1} \partial p_{i_2}\partial p_{i_3}}\\
		+&i_1(i_2+i_3)p_{i_1+i_3}p_{i_2}\frac{\partial^2}{\partial p_{i_1} \partial p_{i_2+i_3}} \\
		+&i_2(i_1+i_3)p_{i_1+i_2}p_{i_3}\frac{\partial^2}{\partial p_{i_2} \partial p_{i_1+i_3}} \\
		+&i_3(i_1+i_2)p_{i_3+i_2}p_{i_1}\frac{\partial^2}{\partial p_{i_3} \partial p_{i_1+i_2}} \\
		+&(i_1+i_2+i_3)p_{i_1}p_{i_2}p_{i_3} \frac{\partial}{\partial p_{i_1+i_2+i_3}}\\
		+&(i_1+i_2+i_3)p_{i_1+i_2+i_3}\frac{\partial}{\partial p_{i_1+i_2+i_3}}) \Phi(g)\\
& = 3 \Delta_3 \Phi(g),
\end{align*}
where the second equality comes from Lemma \ref{3015} and the last equality comes from Definition \ref{37} or Example \ref{38}.

\end{proof}

We now give the extended definition of $\phi$ (Definition \ref{36}) and the construction of new differential operator if we choose arbitrary $d$-cycle (Definition \ref{37}).

\begin{definition}\label{3020}
	Given an $d$-cycle $\beta$ in $S_d$, we define the map $\phi_\beta: S_d \rightarrow S_d$ as
\begin{align*}
\phi_\beta(\delta)=\beta \delta, \quad \delta \in S_d.
\end{align*}
Then, we construct $\Delta_\beta$, which is similar to $\Delta_d$ in Definition \ref{37},
\begin{align*}
		\Delta_\beta =\frac{1}{d} \sum_{\delta \in S_d}\sum_{a_1,...,a_d \geq 1} \hat{p}_{\phi_\beta(\delta)}(a_1,...,a_d) \frac{\partial}{\partial \hat{p}_\delta}(a_1,...,a_d),
	\end{align*}
where we replace $\phi$ by $\phi_\beta$.
\end{definition}

\begin{remark}\label{3021}
	From this definition, it is clear $\Delta_{(3 2 1)}=\Delta_3$.
\end{remark}

\begin{remark}\label{3022}
Recall the first formula in Lemma \ref{3015},
\begin{align*}
i_1 i_2 i_3 \Phi([j_3,j_2,j_1] \alpha)
=\hat{p}_{\phi((1)(2)(3))}(i_1,i_2,i_3) \frac{\partial}{\partial \hat{p}_{(1)(2)(3)}}(i_1,i_2,i_3)\Phi(\alpha).
\end{align*}
Similarly, we can prove
\begin{align*}
i_1 i_2 i_3 \Phi([j_1,j_2,j_3] \alpha)
=\hat{p}_{\phi_\beta((1)(2)(3))}(i_1,i_2,i_3) \frac{\partial}{\partial \hat{p}_{(1)(2)(3)}}(i_1,i_2,i_3)\Phi(\alpha),
\end{align*}
where $\beta=(1 \text{ } 2 \text{ } 3)$. Indeed, the map $\phi_\beta$ corresponds to tuple $[j_1,j_2,j_3]$. We can prove the other formulas in Lemma \ref{3015} similarly.
\end{remark}

\begin{corollary}\label{3033}
	For any $3$-cycle $\beta$, $\Delta_3=\Delta_\beta$ as operators on the ring $\mathbb{C}[p_1,p_2,...]$.
\end{corollary}

\begin{proof}
	Let $\beta=(1 \text{ }2 \text{ }3)$. We have
	\begin{align*}
		\Delta_3 \Phi(g)=&\frac{1}{3} \Phi(\sum_{[j_3,j_2,j_1] \in \bar{C}_{n,3}} [j_3,j_2,j_1]g)\\
		&=\Phi(\bar{K}_{31^{n-3}}g)\\
		&=\Phi(\sum_{[j_1,j_2,j_3] \in \bar{C}_{n,3}} [j_1,j_2,j_3]g)\\
		&=\Delta_\beta \Phi(g),
	\end{align*}
where the last equality comes from Remark \ref{3022}.
	  	
	Hence, $\Delta_\beta=\Delta_3$ as operators on $\mathbb{C}[p_1,p_2,...]$.
\end{proof}

\begin{remark}
	The above argument can be extended to $\Delta_d$, $d \geq 4$, i.e., for any $d$-cycle $\beta$, $\Delta_\beta=\Delta_d$. This will be shown in Corollary \ref{317}.
\end{remark}

\subsection{General Case}
The proof of the general case is very similar to the case $d=3$. First, we generalize Construction \ref{3009}, Definition \ref{311} and \ref{3013}.

\begin{construction}\label{3050}
Let $\bar{\sigma}=[j_d,...,j_1] \in \bar{C}_{n,d}$. We want to classify all permutations $g \in S_n$ according to the occurrence of $j_1,...,j_d$ in the disjoint cycles appearing in $g$. Restrict $g$ to a permutation  $g_{\bar{\sigma}}$ in $S_d$ by forgetting all digits except for $j_1,...,j_d$ but preserving the cycle structure (similar to the construction of $g_{\bar{\sigma}}$ in Construction \ref{3009}). There are $d!$ possible choices for $g_{\bar{\sigma}}$, each of which corresponds to a permutation in $S_d$. By an abuse of the notation, $S_d$ is the permutation group of $\{1,...,d\}$ (Recall Remark \ref{30008} and see Notation \ref{3049}). We say $(g,\bar{\sigma})$ is of type $\tau \in S_d$, if $\tau=g_{\bar{\sigma}} \in S_d$. Clearly, for any element $g \in S_n$, it falls into one and only one case with respect to $\bar{\sigma}$.
\end{construction}

We want to explain the notation $\tau=g_{\bar{\sigma}} \in S_d$ in the above construction.

\begin{notation}\label{3049}
Let $g$ be a permutation in $S_n$, $n\geq d$. We consider two $d$-tuples $\bar{\sigma}=[d,d-1,...,2,1]$ and $\bar{\sigma}'=[j_d,...,j_1]$ in $\bar{C}_{n,d}$. Clearly, $g_{\bar{\sigma}'} \in Aut\{j_d,...,j_1\}$ and $g_{\bar{\sigma}} \in S_d=Aut\{1,2,...,d\}$. But, we want to compare the two permutations in the same permutation group $S_d=Aut\{1,2...,d\}$. Recall the construction in Remark \ref{30008}. Similarly, we construct the bijective map between $\{1,...,d\}$ and $\{j_1,...,j_d\}$ with respect to the order of the integers, which means small integer maps to the small one and larger integer goes to larger one. This map induces an isomorphism $\o: Aut\{j_d,...,j_1\} \rightarrow Aut\{1,...,d\}$. Hence, by an abuse of notations, $g_{\bar{\sigma}'} \in S_d$ means $\o (g_{\bar{\sigma}'}) \in S_d$.
\end{notation}

\begin{definition}\label{3051}
Given any permutation $\alpha \in S_n$ and a positive integer $d$ such that $d \leq n$, we define the map
\begin{align*}
& I_{\alpha,n,d}: \bar{C}_{n,d} \rightarrow \mathbb{Z}^{3}_{> 0}, \\
& I_{\alpha,n,d}([j_d,...,j_1])=(i_d,...,i_1),
\end{align*}
where $i_k=dist(j_k,\alpha,j_1,...,j_d), 1 \leq k \leq d$.
\end{definition}

\begin{definition}\label{3052}
Let $\alpha$ be a permutation in $S_n$. Let $d$ be a positive integer such that $d \leq n$. $i_k$ are positive integers, $1 \leq k \leq d$. Let $\tau$ be a permutation in $S_d$. We define the subset $\bar{C}^{\tau}_{n,d}(\alpha,i_1,...,i_d)$ of $\bar{C}_{n,d}$ as
\begin{align*}
\bar{C}^{\tau}_{n,d}(\alpha,i_1,...,i_d)=\{[j_d,...,j_1] \mid  i_k=dist(j_k,\alpha,j_1,...,j_d), 1 \leq k \leq d, \\ (\alpha,[j_d,...,j_1]) \text{ is of type } \tau \}.
\end{align*}
\end{definition}

\begin{remark}\label{3053}
Let $\alpha$ be a permutation in $S_n$. We have
\begin{align*}
\bar{C}_{n,d}=\bigcup_{\tau \in S_d} \bigcup_{i_1,...,i_d \geq 1} \bar{C}^{\tau}_{n,d}(\alpha,i_1,...,i_d).
\end{align*}
Given any $d$-tuple $[j_d,...,j_1]$, the "distance" $dist(j_i,\alpha,j_1,...,j_d)$ and the type of $(\alpha,[j_d,...,j_1])$ are uniquely determined. Hence, the union above is a disjoint union. Also, there are only finitely many nonempty sets $\bar{C}^{\tau}_{n,d}(\alpha,i_1,...,i_d)$ in the above union.
\end{remark}

\begin{lemma}\label{30539}
Given any two elements $\bar{\sigma}$ and $\bar{\sigma}'$ in $\bar{C}^{\tau}_{n,d}(\alpha,i_1,...,i_d)$, $\bar{\sigma}\alpha$ and $\bar{\sigma}'\alpha$ are of the same type.
\end{lemma}

\begin{proof}
Assume $\bar{\sigma}=[j_d,...,j_1]$ and $\bar{\sigma}'=[j_d',...,j_1']$. By Definition \ref{3052}, $\alpha_{\bar{\sigma}}$ and $\alpha_{\bar{\sigma}'}$ are of the same type $\tau \in S_d$. Hence, $\bar{\sigma}\alpha_{\bar{\sigma}}$ and $\bar{\sigma}'\alpha_{\bar{\sigma}'}$ are of the same type (The second step in Construction \ref{3009}). Also, by definition, we know
\begin{align*}
i_v=dist(j_v,\alpha,j_1,...,j_d)=dist(j'_v,\alpha,j'_1,...,j'_d).
\end{align*}
Hence, $\bar{\sigma}\alpha$ and $\bar{\sigma}'\alpha$ are of the same type.
\end{proof}
\begin{lemma}\label{30540}
Let $\alpha$ be an $n$-cycle in $S_n$. Let $\bar{C}^{\tau}_{n,d}(\alpha,i_1,...,i_d)$ be a nonempty set for some $\tau \in S_d$, and $\bar{\sigma}=[j_d,...,j_1]$ is a $d$-tuple in the set $\bar{C}^{\tau}_{n,d}(\alpha,i_1,...,i_d)$. Then, the number of all elements in this set $\bar{C}^{\tau}_{n,d}(\alpha,i_1,...,i_d)$ is $n$.
\end{lemma}

\begin{proof}
If we want to use $\bar{\sigma}$ to construct some $d$-tuple $[j_d',...,j_1']$ in $\bar{C}^{\tau}_{n,d}(\alpha,i_1,...,i_d)$, we have to pick $d$ integers $j_i'$, $1 \leq i \leq d$ such that
\begin{align*}
i_k=dist(j_k,\alpha,j_1,...,j_d)=dist(j_k',\alpha,j_1',...,j_d').
\end{align*}
At the same time, we know $j_1,...,j_d$ are in the same disjoint cycle and
\begin{align*}
\sum_{k=1}^{d}dist(j_k,\alpha,j_1,...,j_d)=\sum_{k=1}^{d}dist(j_k',\alpha,j_1',...,j_d')=n.
\end{align*}
Hence, the choice of $j_1'$ will completely determine the $d$-tuple $[j_d',...,j_1']$. There are $n$ choices for $j_1'$ and each choice determines a unique $d$-tuple in $\bar{C}^{\tau}_{n,d}(\alpha,i_1,...,i_d)$. It is easy to prove they are all of the elements in $\bar{C}^{\tau}_{n,d}(\alpha,i_1,...,i_d)$. We leave it as an exercise for the reader.
\end{proof}

The following lemma is a generalization of Lemma \ref{30001}.
\begin{lemma}\label{30541}
Let $\alpha$ be a permutation in $S_n$ and let $i_1,...,i_d$ be $d$ positive integers smaller than $n$. Let $\tau$ be a permutation in $S_d$ and $\tau=\tau_1...\tau_m$, which is the product of disjoint cycles of $m$. We define new integers $\widetilde{i}_v$ as
\begin{align*}
\widetilde{i}_v=\sum_{k \in \tau_v} i_k,
\end{align*}
where $1 \leq v \leq m$. We assume the number of disjoint cycles with length $\widetilde{i}_v$ in $\alpha$ is $c_v$, $1 \leq v \leq m$. Then, the number of elements in $\bar{C}^{\tau}_{n,d}(\alpha,i_1,...,i_d)$ is $\prod_{v=1}^m c_v \widetilde{i}_v$.
\end{lemma}
\begin{proof}
If $\bar{\sigma}=[j_d,...,j_1] \in \bar{C}^{\tau}_{n,d}(\alpha,i_1,...,i_d)$, it means $\alpha_{\bar{\sigma}}$ is of type $\tau$ and the integers $j_d,...,j_1$ are chosen from disjoints cycles $\alpha'_v$ with length $i'_v$, $1 \leq v \leq m$, such that
\begin{align*}
i_k=dist(j_k,\alpha,j_1,...,j_d), \quad \widetilde{i}_v=\sum_{k \in \tau_v} i_k.
\end{align*}
By assumption, we know the number of disjoint cycles with length $\widetilde{i}_v$ in $\alpha$ is $c_v$, $1 \leq v \leq m$. Hence, the choice of the disjoint cycles $\alpha'_v$ with length $\widetilde{i}_v$, $1 \leq v \leq m$, is $\prod_{v=1}^m c_v$. Now we fix a possible choice for disjoint cycles $\alpha'_v$, $1 \leq v \leq m$. We want to pick integers $j'_k$, $k \in \tau_v$, from $\alpha'_v$ such that
\begin{align*}
\widetilde{i}_v=\sum_{k \in \tau_v} i'_k, \quad i'_k=dist(j'_k,\alpha,j'_1,...,j'_d)=i_k.
\end{align*}
By Lemma \ref{30540}, the number of choices of picking such integers $j'_k$, $k \in \tau_v$, is $\widetilde{i}_k$, which is the length of $\alpha'_v$. The choices of integers $j_k'$ from different disjoint cycles are independent. Hence, given a possible choice for disjoint cycles $\alpha'_v$, $1 \leq v \leq m$, we can construct $\prod_{v=1}\widetilde{i}_v$ many $d$-tuples in $\bar{C}^{\tau}_{n,d}(\alpha,i_1,...,i_d)$. In conclusion, we can construct $\prod_{v=1}^m c_v \widetilde{i}_v$ many $d$-tuples $\bar{C}^{\tau}_{n,d}(\alpha,i_1,...,i_d)$. It is easy to check they are all elements in $\bar{C}^{\tau}_{n,d}(\alpha,i_1,...,i_d)$.
\end{proof}

\begin{remark}\label{30542}
We use the same notations as in Lemma \ref{30541}. By Definition \ref{32}, we know
\begin{align*}
		\frac{\partial}{\partial \hat{p}_\tau}(i_1,...,i_d) = \prod_{v=1}^{m}((\sum_{k \in \tau_v}i_k) \frac{\partial}{\partial p_{\sum_{k \in \tau_v}i_k}})=\prod_{v=1}^{m}(\widetilde{i}_v \frac{\partial}{\partial p_{\widetilde{i}_v}}).
\end{align*}
\end{remark}

The next lemma is a generalization of Lemma \ref{3015}.
\begin{lemma}\label{3054}
Let $\alpha$ be a permutation in $S_n$ and let $i_1,...,i_d$ be $d$ positive integers smaller than $n$. $\tau$ is a permutation in $S_d$. We have
\begin{align*}
 \Phi(\sum_{[j_d,...,j_1] \in \bar{C}^{\tau}_{n,d}(\alpha,i_1,...,i_d)} [j_d,...,j_1] \alpha)
=\hat{p}_{\phi_d(\tau)}(i_1,...,i_d) \frac{\partial}{\partial \hat{p}_{\tau}}(i_1,...,i_d)\Phi(\alpha).
\end{align*}
\end{lemma}

\begin{proof}
We use the same notations as in Lemma \ref{30541}. If $\bar{C}^{\tau}_{n,d}(\alpha,i_d,...,i_1)$ is empty, with a similar argument as in Lemma \ref{3015}, we can get
\begin{align*}
0=\Phi(\sum_{[j_d,...,j_1] \in \bar{C}^{\tau}_{n,d}(\alpha,i_1,...,i_d)} [j_d,...,j_1] \alpha)
=\hat{p}_{\phi_d(\tau)}(i_1,...,i_d) \frac{\partial}{\partial \hat{p}_{\tau}}(i_1,...,i_d)\Phi(\alpha)=0.
\end{align*}
Hence, we assume there is at least one disjoint cycle with length $\widetilde{i}_v$ in $\alpha$, $1 \leq v \leq m$. The number of disjoint cycles with length $\widetilde{i}_v$ in $\alpha$ is $c_v$. By Lemma \ref{30541}, we know the number of elements in $\bar{C}^{\tau}_{n,d}(\alpha,i_1,...,i_d)$ is $\prod_{v=1}^m c_v \widetilde{i}_v$. By Lemma \ref{30539} and \ref{30541}, we have
\begin{align*}
( \sum_{[j_d,...,j_1] \in \bar{C}^{\tau}_{n,d}(\alpha,i_1,...,i_d)} [j_d,...,j_1] \alpha )=(\prod_{v=1}^m c_v \widetilde{i}_v) \Phi(\alpha'),
\end{align*}
where $\alpha'=[j_d,...,j_1]\alpha$ for some $[j_d,...,j_1] \in \bar{C}^{\tau}_{n,d}(\alpha,i_1,...,i_d)$ (By Lemma \ref{30539}, $\Phi([j_d,...,j_1]\alpha)$ does not depend on the choice of $[j_d,...,j_1]$ in $\bar{C}^{\tau}_{n,d}(\alpha,i_1,...,i_d)$). By assumption, we know there are $c_v$ disjoint cycles with length $\widetilde{i}_v$ in $\alpha$, it means the order of $p_{\widetilde{i}_v}$ in the monomial $\Phi(\alpha)$ is $c_v$. So, when we calculate $\frac{\partial \Phi(\sigma_v)}{\partial p_{\widetilde{i}_v}}$, we will have a coefficient $c_v$. By Remark \ref{30542}, we have
\begin{align*}
\hat{p}_{\phi_d(\tau)}(i_1,...,i_d) \frac{\partial}{\partial \hat{p}_{\tau}}(i_1,...,i_d)\Phi(\alpha)=(\prod_{v=1}^m c_v \widetilde{i}_v) \Phi(\alpha').
\end{align*}
Hence, we get the following equation
\begin{align*}
\Phi(\sum_{[j_d,...,j_1] \in \bar{C}^{\tau}_{n,d}(\alpha,i_1,...,i_d)} [j_d,...,j_1] \alpha)
=\hat{p}_{\phi_d(\tau)}(i_1,...,i_d) \frac{\partial}{\partial \hat{p}_{\tau}}(i_1,...,i_d)\Phi(\alpha).
\end{align*}
\end{proof}

\begin{theorem}\label{315}
	For any $g \in \mathbb{C}S_n$,
	\begin{align*}
		\Phi(K_{1^{n-d}d}g)=\Delta_d \Phi(g).
	\end{align*}
\end{theorem}

\begin{proof}
We assume $g$ is a permutation in $S_n$. Say $g=\alpha$. By Remark \ref{3053}, we have
\begin{align*}
\bar{C}_{n,d}=\bigcup_{\tau \in S_d} \bigcup_{i_1,...,i_d \geq 1} \bar{C}^{\tau}_{n,d}(\alpha,i_1,...,i_d).
\end{align*}
Then, we get
\begin{align*}
\Phi(\bar{K}_{1^{n-d}d}\alpha)&
=\Phi(\sum_{\tau \in S_d} \sum_{i_1,...,i_d \geq 1}\sum_{[j_d,...,j_1] \in \bar{C}^{\tau}_{n,d}(\alpha,i_1,...,i_d)} [j_d,...,j_1] \alpha)\\
& = \sum_{i_1,...,i_d \geq 1}\sum_{\tau \in S_d}\hat{p}_{\phi_d(\tau)}(i_1,...,i_d) \frac{\partial}{\partial \hat{p}_{\tau}}(i_1,...,i_d)\Phi(\alpha)\\
& = d \Delta_d \Phi(\alpha),
\end{align*}
where the second equality comes from Lemma \ref{3054} and the last equality comes from Definition \ref{37}. By Definition \ref{39}, we know the map $\pi_{n,d}:\bar{C}_{n,d} \rightarrow C_{n,d}$ is a $d$-to-$1$ map. So, we have
\begin{align*}
d\Phi(K_{1^{n-d}d}\alpha)=\Phi(\bar{K}_{1^{n-d}d}\alpha) =d\Delta_d \Phi(\alpha).
\end{align*}
\end{proof}

\begin{theorem}\label{316}
	For any positive integer $d$, $\Delta_d = W([d])$ as an operator on $\mathbb{C}[p_1,p_2,...]$.
\end{theorem}

\begin{proof}
	By Theorem \ref{110} and Theorem \ref{315}, it is easy to get this consequence.
\end{proof}

\begin{corollary}\label{317}
	For any $\beta \in S_d$, $\Delta_d=\Delta_\beta$ as operators on $\mathbb{C}[p_1,p_2,...]$.
\end{corollary}

\begin{proof}
	Given any monomial $\prod_{i=1}^k p_{j_i}$ in $\mathbb{C}[p_1,p_2,...]$, where $j_1 \leq j_2 \leq... \leq j_k$, it corresponds to the partition $(j_1,...,j_k)$. We pick a permutation $g$ of type $(j_1,...,j_k)$. Then, we have
	\begin{align*}
		\Delta_d \Phi(g)=&\frac{1}{d} \Phi(\sum_{[j_d,...,j_1] \in \bar{C}_{n,d}} [j_d,...,j_1]g)\\
		&=\frac{1}{d} \Phi(\bar{K}_{1^{n-d}d}g)\\
		&=\Phi(\sum_{[j_{\beta(d)},...,j_{\beta(1)}] \in \bar{C}_{n,d}} [j_{\beta(d)},...,j_{\beta(1)}]g)\\
		&=\Delta_\beta \Phi(g).
	\end{align*}  	
\end{proof}

\begin{corollary}\label{318}
	Let $d_1,d_2$ be positive integers. $W([d_1])$, $W([d_2])$ commutes as operators on $\mathbb{C}[p_1,p_2,...]$, i.e $W([d_1])W([d_2])=W([d_2])W([d_1])$.
\end{corollary}

\begin{proof}
	We take any monomial $\prod_{i=1}^k p_{j_i}$ in the ring $\mathbb{C}[p_1,p_2,...]$. We pick a permutation $g$ corresponding to this monomial. We have
	\begin{align*}
		& W([d_1])W([d_2]) \Phi(g) \\
		= &  \Phi(K_{1^{n-d_1}d_1} K_{1^{n-d_2}d_2}g)\\
		= &  \Phi(K_{1^{n-d_2}d_2} K_{1^{n-d_1}d_1}g) \\
		= &  W([d_2])W([d_1]) \Phi(g).
	\end{align*}
$K_{1^{n-d_{1}}d_{1}}, K_{1^{n-d_{2}}d_2}$ commutes, because they are central element in $\mathbb{C} S_n$. So, we have
\begin{align*}
W([d_1])W([d_2])=W([d_2])W([d_1]).
\end{align*}
\end{proof}

\section{Hurwitz Enumeration Problem}

Suppose $f:X \rightarrow S^2$ is a continuous map and $X$ is a degree $n$ covering of $S_2$ with branched points ${z_1,...,z_k}$. Let $D$ be an open disc such that the branch points are on the boundary of $D$, There are exactly $d$ connected components in $F^{-1}(D)$ which we label from $1$ to $d$. If we focus on a small neighborhood of $z_i$, beginning on the sheet $s$ and going around $z_i$ counter clockwise, we will arrive at a point on another sheet $\pi^{(i)}(s)$. Hence, we construct a permutation $\pi^{(i)}$ for each branch point $z_i$.

Now we choose a point $x \in S^2$ which is not the branch points. We begin at this point $x$ and walk around each branch point as described above. clearly if we begin on the sheet $s$, we must end on this sheet, because the corresponding loop on $S^2/\{z_1,...,z_k\}$ is contractible to a point. So we have the \textit{monodromy condition}
\begin{align*}
\pi^{(1)}...\pi^{(k)}=1.
\end{align*}
Since $X$ is connected, we must be able to move from one sheet to any other, hence, the subgroup generated by $\{\pi^{(1)},...,\pi^{(k)}\}$ must act transitively on the set $\{1,...,n\}$. This is called the \textit{transitivity condition}. Given $k$ partitions $\lambda_i$ of $n$, denote by $Cov_d(\lambda_1,...,\lambda_k)$ the number of $k$-tuples $(\sigma_1,...,\sigma_k) \in S^{k}_{n}$ satisfying the following conditions \cite{Carrel}:
\begin{itemize}
\item $\sigma_i$ is of type $\lambda_i$ for all $i$,
\item $\sigma_1...\sigma_k=1$(the monodromy condition),
\item the subgroup generated by $\{\sigma_1,...,\sigma_k\}$ acts transitive on the set $\{1,...,n\}$.
\end{itemize}
Now, we consider a special case of this problem.

\begin{definition}\label{51}
	Given positive integers $d$, $n$ and $k$, where $d \leq n$, define the number $h^{d}_k(\alpha)$ as following
	\begin{align*}
		h^{[d]}_k(\alpha)=Cov_d(1^{n-d}d,...,1^{n-d}d,\alpha),
	\end{align*}
	where $\alpha$ is a partition of $n$ and there are $k$ copies of partition $(1^{n-d}d)$. Define the generating function
	\begin{align*}
		H^{[d]}(z,p)=H^{[d]}(z,p_1,p_2,...)=\sum_{n \geq 1}\frac{1}{n!}\sum_{k=1}^{\infty}\sum_{\alpha \vdash n}h^{[d]}_k(\alpha)\frac{z^k}{k!}\Phi(\alpha) \text{ .}
	\end{align*}
	Finally, we define the generating function
	\begin{align*}
		\hat{H}^{[d]}=e^{H^{[d]}}.
	\end{align*}	
\end{definition}

\begin{remark}\label{5001}
We expand the generating function $\hat{H}^{[d]}$ with coefficients $\hat{h}^{[d]}_k(\alpha)$,
	\begin{align*}
		\hat{H}^{[d]}(z,p)=\hat{H}^{[d]}(z,p_1,p_2,...)=\sum_{n \geq 1}\frac{1}{n!}\sum_{k=1}^{\infty}\sum_{\alpha \vdash n}\hat{h}^{d}_k(\alpha)\frac{z^k}{k!}\Phi(\alpha)\text{ .}
	\end{align*}
It is easy to check $\hat{h}^{[d]}_k(\alpha)$ is the number of $(k+1)$-tuples $(\sigma_1,...,\sigma_{k},\sigma) \in S^{k+1}_{n}$ satisfying the following conditions:
\begin{itemize}
\item $\sigma_i$ is of type $(1^{n-d}d)$ for all $i$ and $\sigma$ is of type $\alpha$,
\item $\sigma_1...\sigma_k\sigma=1$ (the monodromy condition).
\end{itemize}
Compared with $h^{[d]}_k(\alpha)$, we do not have the transitivity condition.
\end{remark}

\begin{definition}\label{5002}
Given a positive integer $n$, let $\alpha$ be a partition of $n$. We define the set $\mathcal{A}^{[d]} (\alpha,k)$ as $(k+1)$-tuples $(\sigma_1,...,\sigma_k,\sigma) \in S^{k+1}_{n}$ satisfying the two conditions in Remark \ref{5001}, i.e.
\begin{itemize}
\item $\sigma_i$ is of type $(1^{n-d}d)$ for all $i$ and $\sigma$ is of type $\alpha$,
\item $\sigma_1...\sigma_k\sigma=1$ (the monodromy condition).
\end{itemize}
Also, we define another set
\begin{align*}
\tilde{\mathcal{A}}^{[d]}(\alpha,k)=\{(\sigma_2,...,\sigma_k,\sigma) \mid ((\sigma_k...\sigma_2\sigma)^{-1},\sigma_2,...,\sigma_k,\sigma) \in \mathcal{A}^{[d]}(\alpha,k)\}.
\end{align*}
\end{definition}

If $S$ is a finite set, $|S|$ is the cardinality of the set $S$.

\begin{remark}\label{5003}
By the definition of $h^{[d]}_k(\alpha)$, we have
\begin{align*}
\hat{h}^{[d]}_k(\alpha)=|\mathcal{A}^{[d]}(\alpha,k)|=|\tilde{\mathcal{A}}^{[d]}(\alpha,k)|.
\end{align*}
Hence, we can write the generating function $\hat{H}^{[d]}(z,p)$ as
\begin{align*}
\hat{H}^{[d]}(z,p)=\hat{H}^{[d]}(z,p_1,p_2,...)=\sum_{n \geq 1}\frac{1}{n!}\sum_{k=1}^{\infty}\sum_{\alpha \vdash n}|\mathcal{A}^{[d]}(\alpha,k)|\frac{z^k}{k!}\Phi(\alpha)\text{ .}
\end{align*}
\end{remark}

\begin{remark}\label{5004}
Consider the generating series $\hat{H}^{[d]}(z,p)$. Given a specific set $\mathcal{A}^{[d]}(\alpha,k)$, $\alpha \vdash n$, the elements in this set are $(k+1)$-tuples $(\sigma_1,...,\sigma_k,\sigma)$. The parameter corresponding to this set is $\frac{z^k}{k!}\Phi(\alpha)$, where the order of $z$ corresponds to the number of $d$-cycles $k$ and $\Phi(\alpha)$ corresponds to the permutation $\sigma$. We take the sum over all partitions. We will get the "set-valued" generating function
\begin{align*}
\sum_{n \geq 1}\frac{1}{n!}\sum_{k=1}^{\infty}\sum_{\alpha \vdash n}\mathcal{A}^{[d]}(\alpha,k)\frac{z^k}{k!}\Phi(\alpha)\text{ .}
\end{align*}
Since every set is finite, we can take the cardinality of each set, and we get the generating function $\hat{H}^{[d]}(z,p)$ in Definition \ref{51} or Remark \ref{5003}.

Similarly, $\frac{\partial \hat{H}^{[d]}}{\partial z}$ is the generating function for the sets $\tilde{\mathcal{A}}^{[d]}(\alpha,k)$, i.e.
\begin{align*}
\frac{\partial \hat{H}^{[d]}}{\partial z}&=\sum_{n \geq 1}\frac{1}{n!}\sum_{k=1}^{\infty}\sum_{\alpha \vdash n}|\tilde{\mathcal{A}}^{[d]}(\alpha,k)|\frac{z^{k-1}}{(k-1)!}\Phi(\alpha)\\
&=\sum_{n \geq 1}\frac{1}{n!}\sum_{k=1}^{\infty}\sum_{\alpha \vdash n}\hat{h}^{[d]}_k(\alpha)\frac{z^{k-1}}{(k-1)!}\Phi(\alpha)\text{ .}
\end{align*}
\end{remark}

\begin{definition}\label{5005}
Let $k,n,d$ be three positive integers, where $n \geq d$. We define the set $\mathcal{A}^{[d]}(k,n)$ as following
\begin{align*}
\mathcal{A}^{[d]}(k,n)=\bigcup_{\alpha \vdash n} \mathcal{A}^{[d]} (k ,\alpha).
\end{align*}
The union is disjoint.
\end{definition}

\begin{lemma}\label{5006}
Let $k,n,d$ be three positive integers, where $n \geq d$. We have
\begin{align*}
\sum_{\alpha \vdash n} h^{[d]}_k(\alpha) \Phi(\alpha)= \sum_{\alpha' \vdash n}h^{[d]}_{k-1}(\alpha') \Phi(K_{1^{n-d}d}\alpha').
\end{align*}
\end{lemma}

\begin{proof}
First, we consider about the sets $\mathcal{A}^{[d]}(k,n)$ and $\mathcal{A}^{[d]}(k-1,n)$. Given any element $(\sigma_1,...,\sigma_k,\sigma) \in \mathcal{A}^{[d]}(k,n)$, it corresponds to a unique element $(\sigma_2,...,\sigma_k,\sigma') \in \mathcal{A}^{[d]}(k-1,n)$, where $\sigma_2...\sigma_k\sigma'=1$. Now given any element $(\sigma_2,...,\sigma_k,\sigma') \in \mathcal{A}^{[d]}(k-1,n)$ and any $d$-cycle $\sigma_1$, we can construct an element $(\sigma_1,...,\sigma_k,\sigma) \in \mathcal{A}^{[d]}(k,n)$, where $\sigma_1...\sigma_k\sigma=1$. Indeed, we can construct different elements in $\mathcal{A}^{[d]}(k,n)$ by adding different $d$-cycles $\sigma_1$ to $(\sigma_2,...,\sigma_k,\sigma')$. The number of elements we construct from the element $(\sigma_2,...,\sigma_k,\sigma')$ is $\frac{1}{d}{n \choose d}d!$, where $\frac{1}{d}{n \choose d}d!$ is the number of $d$-cycles in $S_n$. From the discussion, we can get all elements in $\mathcal{A}^{[d]}(k,n)$ by adding different $d$-cycles to elements in $\mathcal{A}^{[d]}(k-1,n)$. Also, we have
\begin{align*}
|\mathcal{A}^{[d]}(k,n)|=\frac{1}{d}{n \choose d}d!|\mathcal{A}^{[d]}(k-1,n)|.
\end{align*}
Recall the definition of $\mathcal{A}^{[d]}(k,n)$,
\begin{align*}
\mathcal{A}^{[d]}(k,n)=\bigcup_{\alpha \vdash n} \mathcal{A}^{[d]} (k ,\alpha).
\end{align*}
Hence, we have the following formula
\begin{align*}
\sum_{\alpha \vdash n} h^{[d]}_k(\alpha) \Phi(\alpha)= \sum_{\alpha' \vdash n}h^{[d]}_{k-1}(\alpha') \Phi(K_{1^{n-d}d}\alpha').
\end{align*}
\end{proof}

\begin{theorem}\label{52}
	$\hat{H}^{[d]}$ is the solution to the differential equation
	\begin{align*}
		\frac{\partial \hat{H}^{[d]}}{\partial z} = W([d]) \hat{H}^{[d]}
	\end{align*}
	with initial condition
	\begin{align*}
		\hat{H}^{[d]}(0,p)=e^{p_1}
	\end{align*}
\end{theorem}

\begin{proof}
\begin{align*}
\frac{\partial \hat{H}^{[d]}}{\partial z} &=\sum_{n \geq 1}\frac{1}{n!}\sum_{k=1}^{\infty}\sum_{\alpha \vdash n}\frac{z^{k-1}}{(k-1)!}h^{[d]}_k(\alpha)\Phi(\alpha) \\
&=\sum_{n \geq 1}\frac{1}{n!}\sum_{k=1}^{\infty}\sum_{\alpha \vdash n}\frac{z^{k-1}}{(k-1)!}h^{[d]}_{k-1}(\alpha) \Phi(K_{1^{n-d}d}\alpha) \\
&=\sum_{n \geq 1}\frac{1}{n!}\sum_{k=1}^{\infty}\sum_{\alpha \vdash n}\frac{z^{k-1}}{(k-1)!}h^{[d]}_{k-1}(\alpha) W([d])\Phi(\alpha) \\
&=W([d]) \hat{H}^{[d]},
\end{align*}
where the first equality comes from \ref{5004}, the second equality is the consequence of Lemma \ref{5006} and the last equality comes from Theorem \ref{112}. It is easy to check the initial condition.
\end{proof}


\begin{thebibliography}{9}
\bibitem{Carrel}
S. Carrel, \textit{Combinatorics and the {KP} Hierarchy}, preprint (2009), available at
	\texttt{https://uwspace.uwaterloo.ca/handle/10012/4770}.
\bibitem{Cre}
M. Crescimanno and W. Taylor, \textit{Large N Phases of Chiral $\text{QCD}_2$}, Nuclear Phys. B (1995).

\bibitem{Fulton}
W. Fulton and J. Harris, \textit{Representation Theory}, Springer, 1991.

\bibitem{MR1249468}
Goulden, I. P., \textit{A differential operator for symmetric functions and the
	combinatorics of multiplying transpositions}, Trans. Amer. Math. Soc. (1994), 421-440.
\bibitem{MR1396978}
Goulden, I. P. and Jackson, D. M., \textit{Transitive factorisations into transpositions and holomorphic
	mappings on the sphere}, Proc. Amer. Math. Soc. (1997), 51-60.
\bibitem{Kac}
V. G. Kac and A. K. Raina, \textit{Bombay Lectures on Highest Weight Representations of Infinite Dimensional {L}ie Algebras}, World Scienfitic Publishing Co.Pte.Ltd, 1988.

\bibitem{Lando}
S. K. Lando and A. K. Zvonkin, \textit{Graphs on Surfaces and their Applications}, Springer, 2004.

\bibitem{MacDonald}
J. G. MacDonald, \textit{Symmetric Functions and {H}all Polynomials}, Oxford University Press 1995.

\bibitem{MR2864467}
A. Mironov, A. Morosov and S. Natanzon, \textit{Algebra of differential operators associated with {Y}oung
	diagrams}, J. Geom. Phys. (2012), 148--155.

\bibitem{Mir}
A. Mironov, A. Morosov and S. Natanzon, \textit{Complete Set of Cut-and-Join Operators in Hurwitz-Kontsevich Theory}, arXiv:0904.4227.

\bibitem{Sun1}
H. Sun, \textit{W-Operators and Permutation Groups}, arXiv:1610.06624.
\end{thebibliography}
\end{document}